\documentclass[11pt,reqno]{amsart}
\usepackage{amssymb, stmaryrd}
\usepackage{array}
\usepackage{graphicx}
\usepackage{caption, booktabs}
\usepackage[all]{xy}
\usepackage{comment}
\usepackage{cite}
\usepackage[left=2.5cm,right=2.5cm,top=3cm,bottom=3cm]{geometry}
\usepackage{color}
\usepackage{tikz}
\usepackage{tikz-cd}

\makeatletter
\@namedef{subjclassname@1991}{2020 Mathematics Subject Classification}
\makeatother

\newcommand{\C}{\mathbb{C}}

\newcommand{\g}{\mathfrak{g}}

\newcommand{\p}{\mathfrak{p}}
\newcommand{\uu}{\mathfrak{u}}
\newcommand{\bb}{\mathfrak{b}}
\newcommand{\lf}{\mathfrak{l}}

\newcommand{\ad}{\mathrm{ad}}
\newcommand{\Ad}{\mathrm{Ad}}

\numberwithin{equation}{section}

\newtheorem{thm}{Theorem}[section]

\newtheorem{lem}[thm]{Lemma}
\newtheorem{cor}[thm]{Corollary}
\newtheorem{prop}[thm]{Proposition}
\theoremstyle{definition}
\newtheorem{rem}[thm]{Remark}

\newtheorem{conj}[thm]{Conjecture}

\title[Hessenberg varieties and Poisson slices]{Hessenberg varieties and Poisson slices}

\author[Peter Crooks]{Peter Crooks}
\author[Markus R\"oser]{Markus R\"oser}
\address[Peter Crooks]{Department of Mathematics \\ Northeastern University \\ 360 Huntington Avenue \\ Boston, MA 02115, USA}
\email{p.crooks@northeastern.edu}
\address[Markus R\"oser]{Fachbereich Mathematik\\ Universit\"at Hamburg\\ Bundesstra\ss e 55 (Geomatikum) \\ 20146 Hamburg\\ Germany}
\email{markus.roeser@uni-hamburg.de}

\begin{document}

\date{}

\dedicatory{}

\commby{}

\begin{abstract}
This work pursues a circle of Lie-theoretic ideas involving Hessenberg varieties, Poisson geometry, and wonderful compactifications. In more detail, one may associate a symplectic Hamiltonian $G$-variety $\mu:G\times\mathcal{S}\longrightarrow\mathfrak{g}$ to each complex semisimple Lie algebra $\mathfrak{g}$ with adjoint group $G$ and fixed Kostant section $\mathcal{S}\subseteq\mathfrak{g}$. This variety is one of Bielawski's hyperk\"ahler slices, and it is central to Moore and Tachikawa's work on topological quantum field theories. It also bears a close relation to two log symplectic Hamiltonian $G$-varieties $\overline{\mu}_{\mathcal{S}}:\overline{G\times\mathcal{S}}\longrightarrow\mathfrak{g}$ and $\nu:\mathrm{Hess}\longrightarrow\mathfrak{g}$. The former is a Poisson transversal in the log cotangent bundle of the wonderful compactification $\overline{G}$, while the latter is the standard family of Hessenberg varieties. Each of $\overline{\mu}$ and $\nu$ is known to be a fibrewise compactification of $\mu$.

We exploit the theory of Poisson slices to relate the fibrewise compactifications mentioned above. Our main result is a canonical $G$-equivariant bimeromorphism $\mathrm{Hess}\cong\overline{G\times\mathcal{S}}$ of varieties over $\mathfrak{g}$. This bimeromorphism is shown to be a Hamiltonian $G$-variety isomorphism in codimension one, and to be compatible with a Poisson isomorphism obtained by B\u{a}libanu. We also show our bimeromorphism to be a biholomorphism if $\mathfrak{g}=\mathfrak{sl}_2$, and we conjecture that this is the case for arbitrary $\mathfrak{g}$. We conclude by discussing the implications of our conjecture for Hessenberg varieties.
\end{abstract}  

\maketitle
\begin{scriptsize}
\tableofcontents
\end{scriptsize}
\section{Introduction}

\subsection{Context}
Let $\g$ be a complex semisimple Lie algebra with adjoint group $G$ and fixed Kostant section $\mathcal{S}\subseteq\g$. These data determine a symplectic Hamiltonian $G$-variety $\mu:G\times\mathcal{S}\longrightarrow\g$ that has received considerable attention in the literature; it is perhaps the most fundamental of Bielawski's \textit{hyperk\"ahler slices} \cite{Bielawski,BielawskiComplex,CrooksVP}, and it features prominently in Moore and Tachikawa's work on topological quantum field theories \cite{Moore}. Some more recent work is concerned with the moment map $\mu:G\times\mathcal{S}\longrightarrow\g$ and its geometric properties \cite{CrooksRayan,AbeCrooks,CrooksBulletin}. One such property is the fact that $\mu$ admits two fibrewise compactifications; the first is a well-studied family of Hessenberg varieties, while the other is constructed via the wonderful compactification $\overline{G}$ of $G$. The following discussion gives some context for the two compactifications.  

Hessenberg varieties arise as a natural generalization of Grothendieck--Springer fibres, and their study is central to modern research in algebraic geometry \cite{AbeSelecta,Brosnan,PrecupTransformation,TymoczkoSelecta,DeMari}, combinatorics \cite{AbeCrooks1,AbeCombinatorial,TymoczkoContemp,Shareshian,HaradaPrecup}, representation theory \cite{BalibanuPeterson,BalibanuUniversal,BalibanuCrooks}, and symplectic geometry \cite{AbeCrooks,KostantLett}. One begins by fixing a Borel subgroup $B\subseteq G$ with Lie algebra $\mathfrak{b}\subseteq\mathfrak{g}$. A \textit{Hessenberg subspace} is then defined to be a $B$-invariant vector subspace $H\subseteq\mathfrak{g}$ that contains $\mathfrak{b}$. Each Hessenberg subspace $H\subseteq\mathfrak{g}$ determines a $G$-equivariant vector bundle $G\times_B H\longrightarrow G/B$, and the total space of this bundle is Poisson. The $G$-action on $G\times_B H$ is Hamiltonian and admits an explicit moment map $\nu_{H}:G\times_B H\longrightarrow\mathfrak{g}$. One calls 
$$\mathrm{Hess}(x,H):=\nu_{H}^{-1}(x)$$
the \textit{Hessenberg variety} associated to $H$ and $x\in\g$, and regards $\nu_{H}$ as the family of all Hessenberg varieties associated to $H$.    

The so-called \textit{standard Hessenberg subspace} is the annihilator $\mathfrak{m}$ of $[\mathfrak{u},\mathfrak{u}]$ under the Killing form, where $\mathfrak{u}$ is the nilradical of $\mathfrak{b}$. The resulting family $\nu:=\nu_{\mathfrak{m}}:G\times_B \mathfrak{m}\longrightarrow\mathfrak{g}$ is of particular importance. One reason is that the fibres of $\nu$ appear in interesting contexts; $\mathrm{Hess}(x,\mathfrak{m})$ is isomorphic to the Peterson variety if $x\in\mathfrak{g}$ is regular and nilpotent, while $\mathrm{Hess}(x,\mathfrak{m})$ is a well-studied smooth projective toric variety if $x$ is regular and semisimple. A second reason is that $G\times_B\mathfrak{m}$ enjoys a rich Poisson geometry; B\u{a}libanu has shown this variety to be log symplectic \cite{BalibanuUniversal}, while Abe and the first author have identified $G\times\mathcal{S}$ as the unique open dense symplectic leaf in $G\times_B\mathfrak{m}$ \cite{AbeCrooks}. It is in this context that one realizes $\nu:G\times_B\mathfrak{m}\longrightarrow\g$ as a fibrewise compactification of $\mu:G\times\mathcal{S}\longrightarrow\g$.

The connection between $\mu:G\times\mathcal{S}\longrightarrow\g$ and $\overline{G}$ is realized via $T^*\overline{G}(\log D)$, the \textit{log cotangent bundle} of $\overline{G}$. This variety is log symplectic, and it contains a distinguished log symplectic subvariety $\overline{G\times\mathcal{S}}\subseteq T^*\overline{G}(\log D)$. The subvariety $\overline{G\times\mathcal{S}}$ is a so-called \textit{Poisson slice} \cite{CrooksRoeserSlice} carrying a Hamiltonian $G$-action and moment map $\overline{\mu}:\overline{G\times\mathcal{S}}\longrightarrow\g$, and it contains $G\times\mathcal{S}$ as its unique open dense symplectic leaf. The moment map $\overline{\mu}:\overline{G\times\mathcal{S}}\longrightarrow\g$ is thereby a fibrewise compactification of $\mu:G\times\mathcal{S}\longrightarrow\g$.

It is natural to seek a relationship between the log symplectic fibrewise compactifications $\nu:G\times_B\mathfrak{m}\longrightarrow\g$ and $\overline{\mu}:\overline{G\times\mathcal{S}}\longrightarrow\g$ of $\mu:G\times\mathcal{S}\longrightarrow\g$. We are thereby drawn to B\u{a}libanu's recent article \cite{BalibanuUniversal}, which studies the \textit{universal centralizer} $\mathcal{Z}_{\mathfrak{g}}\longrightarrow\mathcal{S}$. B\u{a}libanu takes an appropriate Whittaker reduction of $T^*\overline{G}(\log(D))$ and obtains a log symplectic fibrewise compactification $\overline{\mathcal{Z}_{\mathfrak{g}}}\longrightarrow\mathcal{S}$ of the universal centralizer. She subsequently shows $\nu^{-1}(\mathcal{S})$ to be Poisson, and to be isomorphic to $\overline{\mathcal{Z}_{\mathfrak{g}}}$ as a Poisson variety over $\mathcal{S}$.

\subsection{Summary of results}
We use the theory of Poisson slices \cite{CrooksRoeserSlice} to construct a canonical bimeromorphism $\overline{G\times\mathcal{S}}\cong G\times_B\mathfrak{m}$ and explain its relation to B\u{a}libanu's Poisson isomorphism. To this end, fix a Borel subgroup $B$ with Lie algebra $\mathfrak{b}$. Let us also fix a principal $\mathfrak{sl}_2$-triple $\tau=(\xi,h,\eta)\in\g^{\oplus 3}$ with $\eta\in\mathfrak{b}$. One then has an associated Slodowy slice (a.k.a. Kostant section) $$\mathcal{S}:=\xi+\g_{\eta}\subseteq\g.$$ The following is our main result. 

\begin{thm}\label{Theorem: Main theorem}
The following statements hold.
\begin{itemize}
\item[(i)] There exists a Hamiltonian $G$-variety isomorphism
$$\overline{G\times\mathcal{S}}^{\circ}\overset{\cong}\longrightarrow (G\times_B\mathfrak{m})^{\circ},$$
where $\overline{G\times\mathcal{S}}^{\circ}\subseteq\overline{G\times\mathcal{S}}$ and $(G\times_B\mathfrak{m})^{\circ}\subseteq G\times_B\mathfrak{m}$ are explicit $G$-invariant open dense subvarieties with complements of codimension at least two.
\item[(ii)] The isomorphism in $\mathrm{(i)}$ restricts to a variety isomorphism
$$\overline{G\times\mathcal{S}}^{\circ}\cap\overline{\mathcal{Z}_{\mathfrak{g}}}\overset{\cong}\longrightarrow (G\times_B\mathfrak{m})^{\circ}\cap\nu^{-1}(\mathcal{S}).$$ This restricted isomorphism agrees with the restriction of B\u{a}libanu's Poisson isomorphism to $\overline{G\times\mathcal{S}}^{\circ}\cap\overline{\mathcal{Z}_{\mathfrak{g}}}$.
\item[(iii)] The isomorphism in $\mathrm{(i)}$ extends to a $G$-equivariant bimeromorphism
$$\overline{G\times\mathcal{S}}\overset{\cong}\longrightarrow G\times_B\mathfrak{m}$$ for which
$$\begin{tikzcd}
\overline{G\times\mathcal{S}} \arrow{rr}[above]{\cong} \arrow[swap]{dr}{\overline{\mu}_{\mathcal{S}}}& & G\times_B\mathfrak{m} \arrow{dl}{\nu}\\
& \mathfrak{g} & 
\end{tikzcd}$$
commutes.
\item[(iv)] The bimeromorphism in $\mathrm{(iii)}$ is a biholomorphism if $\g=\mathfrak{sl}_2$. 
\end{itemize}
\end{thm}

This theorem leads to the following conjecture.

\begin{conj}\label{Conj}
The bimeromorphism in Theorem \ref{Theorem: Main theorem}(iii) is a biholomorphism.
\end{conj}

We show this conjecture to have the following implications for Hessenberg varieties.

\begin{cor}\label{Corollary: Main corollary}
Assume that Conjecture \ref{Conj} is true, and let $\overline{\rho}:\overline{G\times\mathcal{S}}\longrightarrow G\times_B\mathfrak{m}$ be the biholomorphism from Theorem \ref{Theorem: Main theorem}$\mathrm{(iii)}$. Let $\pi:T^*\overline{G}(\log(D))\longrightarrow\overline{G}$ be the bundle projection map. If $x\in\mathfrak{g}$, then the composite map
$$\mathrm{Hess}(x,\mathfrak{m})\overset{\overline{\rho}^{-1}}\longrightarrow\overline{\mu}_{\mathcal{S}}^{-1}(x)\overset{\pi}\longrightarrow\overline{G}$$ is a $G_x$-equivariant closed embedding of algebraic varieties.
\end{cor}

More succinctly, the truth of Conjecture \ref{Conj} would imply that every Hessenberg variety in the family $\nu:G\times_B\mathfrak{m}\longrightarrow\g$ admits an equivariant closed embedding into $\overline{G}$.

\subsection{Organization}
Section \ref{Section: Some preliminaries} assembles some of the Lie-theoretic and Poisson-geometric facts, conventions, and notation underlying this paper.  The heart of our paper begins in Section \ref{Section: Poisson slices in the logarithmic cotangent bundle}, which is principally concerned with the log symplectic variety $\overline{G\times\mathcal{S}}$ and its properties. Section \ref{Section: Relation} expands our discussion to include implications for the family $\nu:G\times_B\mathfrak{m}\longrightarrow\g$. This section contains the proofs of Theorem \ref{Theorem: Main theorem} and Corollary \ref{Corollary: Main corollary}, as well as the requisite machinery. A brief list of recurring notation appears after Section \ref{Section: Relation}.  

\subsection*{Acknowledgements}
We gratefully acknowledge Ana B\u{a}libanu for constructive conversations and suggestions. The first author is supported by an NSERC Postdoctoral Fellowship [PDF--516638]. 

\section{Some preliminaries}\label{Section: Some preliminaries}
This section gathers some of the notation, conventions, and standard facts used throughout our paper. Our principal objective is to outline the Lie-theoretic constructions relevant to later sections.   

\subsection{Fundamental conventions}
This paper works exclusively over $\mathbb{C}$. We understand \textit{group action} as meaning \textit{left group action}. The dimension of an algebraic variety is the supremum of the dimensions of its irreducible components.
We use the term \textit{smooth variety} in reference to a pure-dimensional algebraic variety $X$ satisfying $\dim(T_xX)=\dim X$ for all $x\in X$.

\subsection{Lie-theoretic conventions}\label{Subsection: Lie-theoretic conventions}
Let $\g$ be a finite-dimensional, rank-$\ell$, semisimple Lie algebra with adjoint group $G$. Write 
$$\Ad:G\longrightarrow\operatorname{GL}(\mathfrak{g}),\quad g\mapsto \Ad_g,\quad g\in G$$ for the adjoint representation of $G$ on $\mathfrak{g}$, and  
$$\ad:\g\longrightarrow\mathfrak{gl}(\mathfrak{g}),\quad x\mapsto \ad_x,\quad x\in\mathfrak{g}$$ for the adjoint representation of $\mathfrak{g}$ on itself.
Each $x\in\mathfrak{g}$ admits a $G$-stabilizer
$$G_x:=\{g\in G:\mathrm{Ad}_g(x)=x\}$$ and a $\mathfrak{g}$-centralizer
$$\mathfrak{g}_x:=\mathrm{ker}(\mathrm{ad}_x)=\{y\in\mathfrak{g}:[x,y]=0\}.$$
One also has the $G$-invariant open dense subvariety
$$\mathfrak{g}^{\text{r}}:=\{x\in\mathfrak{g}:\dim(\mathfrak{g}_x)=\ell\}$$ of all regular elements in $\mathfrak{g}$.

One calls $x\in \g$ \textit{semisimple} (resp. \textit{nilpotent}) if $\ad_x\in\mathfrak{gl}(\g)$ is diagonalizable (resp. nilpotent) as a vector space endomorphism. Let $\mathfrak{g}^{\text{rs}}$ denote the $G$-invariant open dense subvariety of regular semisimple elements in $\g$. We also set
$$V^{\text{r}}:=V\cap\mathfrak{g}^{\text{r}}$$ for any subset $V\subseteq\g$.

Let $\langle\cdot,\cdot\rangle:\mathfrak{g}\otimes_{\mathbb{C}}\mathfrak{g}\longrightarrow\mathbb{C}$ denote the Killing form, and write $V^{\perp}\subseteq\g$ for the annihilator of a subspace $V\subseteq\g$ under this form. The Killing form is non-degenerate and $G$-invariant, implying that
\begin{equation}\label{Equation: Killing}\mathfrak{g}\longrightarrow\mathfrak{g}^*,\quad x\mapsto\langle x,\cdot\rangle,\quad x\in\mathfrak{g}\end{equation} 
and
$$\g\oplus\g\longrightarrow (\g\oplus\g)^*,\quad (x_1,x_2)\mapsto (\langle x_1,\cdot\rangle,-\langle x_2,\cdot\rangle),\quad (x_1,x_2)\in\g\oplus\g$$
define $G$-module and ($G\times G$)-module isomorphisms, respectively. We use these isomorphisms to freely identify $\mathfrak{g}$ with $\mathfrak{g}^*$ and $\g\oplus\g$ with $(\g\oplus\g)^*$ throughout the paper. The canonical Poisson structure on $\mathfrak{g}^*$ then renders $\mathfrak{g}$ a Poisson variety. This endows the coordinate algebra $\mathbb{C}[\mathfrak{g}]=\mathrm{Sym}(\mathfrak{g}^*)$ with a Poisson bracket, defined as follows:
$$\{f_1,f_2\}(x)=\langle x,[(df_1)_x,(df_2)_x]\rangle$$ for all $f_1,f_2\in\mathbb{C}[\mathfrak{g}]$ and $x\in\mathfrak{g}$, where $(df_1)_x,(df_2)_x\in\mathfrak{g}^*$ are regarded as elements of $\mathfrak{g}$ via \eqref{Equation: Killing}.
We also note that the symplectic leaves of $\mathfrak{g}$ are precisely the adjoint orbits
$$Gx:=\{\mathrm{Ad}_g(x):g\in G\},\quad x\in\mathfrak{g}.$$

We henceforth fix a principal $\mathfrak{sl}_2$-triple $\tau$, i.e. $\tau=(\xi,h,\eta)\in\g^{\oplus 3}$ with $\xi,h,\eta\in\mathfrak{g}^{\text{r}}$ and
$$[\xi,\eta]=h,\quad [h,\xi]=2\xi,\quad\text{and}\quad [h,\eta]=-2\eta.$$
Let $\mathfrak{b}\subseteq\g$ be the unique Borel subalgebra that contains $\eta$. The Cartan subalgebra $\mathfrak{t}:=\g_h$ then satisfies $\mathfrak{t}\subseteq\mathfrak{b}$, and this gives rise to collections of roots $\Phi\subseteq\mathfrak{t}^*$, positive roots $\Phi^+\subseteq\Phi$, negative roots $\Phi^-=-\Phi^+$, and simple roots $\Pi\subseteq\Phi^+$. Each subset $I\subseteq\Pi$ then determines parabolic subalgebras $$\mathfrak{p}_I:=\mathfrak{b}\oplus\bigoplus_{\alpha\in\Phi_I^{-}}\mathfrak{g}_{\alpha}\quad\text{and}\quad\mathfrak{p}_I^{-}=\mathfrak{b}^-\oplus\bigoplus_{\alpha\in\Phi_I^+}\mathfrak{g}_{\alpha},$$
where $\mathfrak{b}^{-}\subseteq\g$ is the Borel subalgebra opposite to $\mathfrak{b}$ with respect to $\mathfrak{t}$ and $\Phi_I^{+}$ (resp. $\Phi_I^-$) is the set of positive (resp. negative) roots in the $\mathbb{Z}$-span of $I$. Note that $$\mathfrak{l}_I:=\mathfrak{p}_I\cap\mathfrak{p}_I^{-}$$ is a Levi subalgebra of $\g$. Let $\mathfrak{u}_I$ and $\mathfrak{u}_I^-$ denote the nilradicals of $\mathfrak{p}_I$ and $\mathfrak{p}_I^{-}$, respectively, observing that
$$\mathfrak{p}_I=\mathfrak{l}_I\oplus\mathfrak{u}_I\quad\text{and}\quad\mathfrak{p}_I^{-}=\mathfrak{l}_I\oplus\mathfrak{u}_I^-.$$
We have $\mathfrak{p}_{\emptyset}=\mathfrak{b}$, $\mathfrak{p}_{\emptyset}^{-}=\mathfrak{b}^{-}$, and $\mathfrak{l}_{\emptyset}=\mathfrak{t}$, and we adopt the notation
$$\quad\mathfrak{u}:=\mathfrak{u}_{\emptyset}\quad\text{and}\quad\mathfrak{u}^{-}:=\mathfrak{u}_{\emptyset}^{-}.$$  

We conclude with a brief discussion of invariant theory. To this end, consider the finitely-generated subalgebra of $\mathbb{C}[\g]$ given by
$$\mathbb{C}[\g]^G:=\{f\in\mathbb{C}[\g]:f(\mathrm{Ad}_g(x))=f(x)\text{ for all }g\in G\text{ and }x\in\g\}.$$ Denote by 
$$\chi:\g\longrightarrow\mathrm{Spec}(\mathbb{C}[\g]^G)$$
the morphism of affine varieties corresponding to the inclusion $\mathbb{C}[\g]^G\subseteq\mathbb{C}[\g]$. This morphism is called the \textit{adjoint quotient} of $\g$.

Now consider the Slodowy slice $$\mathcal{S}:=\xi+\g_{\eta}$$ determined by our principal $\mathfrak{sl}_2$-triple $\tau$. This slice consists of regular elements, and it is a fundamental domain for the adjoint action of $G$ on $\mathfrak{g}^{\text{r}}$ (see \cite[Theorem 8]{KostantLie}). This slice is also called a \textit{Kostant section}, reflecting the fact that
$$\chi\big\vert_{\mathcal{S}}:\mathcal{S}\longrightarrow\mathrm{Spec}(\mathbb{C}[\mathfrak{g}]^G)$$ is an isomorphism (see \cite[Theorem 7]{KostantLie}). We may therefore consider
$$x_{\mathcal{S}}:=(\chi\big\vert_{\mathcal{S}})^{-1}(\chi(x))$$ for each $x\in\mathfrak{g}$, i.e. $x_{\mathcal{S}}$ is the unique point at which $\mathcal{S}$ meets $\chi^{-1}(\chi(x))$. It is known that
$$\chi^{-1}(\chi(x))=\overline{Gx_{\mathcal{S}}}\quad\text{and}\quad\chi^{-1}(\chi(x))\cap\mathfrak{g}^{\text{r}}=Gx_{\mathcal{S}}$$ for all $x\in\mathfrak{g}$, and that $$\chi^{-1}(\chi(x))=Gx=Gx_{\mathcal{S}}$$ for all $x\in\mathfrak{g}^{\text{rs}}$ (see \cite[Theorem 3]{KostantLie}). 
 
\subsection{Poisson varieties}
Let $X$ be a smooth variety with structure sheaf $\mathcal{O}_X$ and tangent bundle $TX$. Suppose that $P$ is a global section of $\Lambda^2(TX)$, and consider the bracket on $\mathcal{O}_X$ defined by
$$\{f_1,f_2\}:=P(df_1\wedge df_2)$$
for all open subsets $U\subseteq X$ and functions $f_1,f_2\in\mathcal{O}_X(U)$. We refer to $(X,P)$ as a \textit{Poisson variety} if the aforementioned bracket renders $\mathcal{O}_X$ a sheaf of Poisson algebras. In this case, $P$ is called the \textit{Poisson bivector}. We shall always understand $(X_1\times X_2,P_1\oplus(-P_2))$ as being the product of the Poisson varieties $(X_1,P_1)$ and $(X_2,P_2)$. We also note that every symplectic variety is canonically a Poisson variety.

Let $(X,P)$ be a Poisson variety. We may contract $P$ with covectors and realize it as a bundle morphism $$P:T^*X\longrightarrow TX,$$ whose image is a holomorphic distribution on $X$. The maximal integral submanifolds of this distribution are called the \textit{symplectic leaves} of $X$, and each comes equipped with a holomorphic symplectic form. The symplectic form $\omega_L$ on a symplectic leaf $L\subseteq X$ is constructed as follows. One defines the \textit{Hamiltonian vector field} of a locally defined function $f\in\mathcal{O}_X$ by
$$H_f:=-P(df).$$ The gives rise to the tangent space description
$$T_xL=\{(H_f)_x:f\in\mathcal{O}_X\}$$
for all $x\in L$, and one has
$$(\omega_L)_x((H_{f_1})_x,(H_{f_2})_x)=\{f_1,f_2\}(x)$$ for all $x\in L$ and $f_1,f_2\in\mathcal{O}_X$ defined near $x$.

We now briefly discuss the notion of a \textit{log symplectic variety} \cite{Gualtieri,GualtieriPym,PymSchedler,Goto,Radko,GuilleminAdv,GMP}, a slight variant of a symplectic variety in the Poisson category. To this end, let $(X,P)$ be a Poisson variety with a unique open dense symplectic leaf. It follows that $\dim X=2n$ for some non-negative integer $n$, and one may consider the section $P^n$ of $\Lambda^{2n}(TX)$. One calls $(X,P)$ a \textit{log symplectic variety} if the vanishing locus of $P^n$ is a reduced normal crossing divisor $Y\subseteq X$. In this case, we call $Y$ \textit{the divisor} of $(X,P)$.

\subsection{Poisson slices}\label{Subsection: Poisson slices}
Let $(X,P)$ be a Poisson variety. A smooth locally closed subvariety $Z\subseteq X$ is called a \textit{Poisson transversal} if
$$T_zX=T_zZ\oplus P_z((T_zZ)^{\dagger})$$ for all $z\in Z$, where $(T_zZ)^{\dagger}\subseteq T_z^*X$ is the annihilator of $T_zZ$. This decomposition induces an inclusion $T^*Z\subseteq T^*X$, and one has
$P(T^*Z)\subseteq TZ\subseteq TX$. The restriction
$$P_Z:=P\big\vert_{T^*Z}:T^*Z\longrightarrow TZ$$ is then a Poisson bivector on $Z$. We regard all Poisson transversals as coming equipped with the aforementioned Poisson variety structure.

Now suppose that $(X,P,\nu)$ is a Hamiltonian $G$-variety, i.e. $(X,P)$ is Poisson variety endowed with a Hamiltonian $G$-action and moment map $\nu:X\longrightarrow\g$. Let us also recall the Slodowy slice $\mathcal{S}\subseteq\g$ fixed in Section \ref{Subsection: Lie-theoretic conventions}. The preimage $\nu^{-1}(\mathcal{S})$ is a Poisson transversal in $X$ and an example of a Poisson slice \cite{CrooksRoeserSlice}. Some salient properties of this slice appear in the following specialized version of \cite[Corollary 3.7]{CrooksRoeserSlice}.

\begin{prop}\label{Proposition: Utility}
Suppose that $(X,P,\nu)$ is a Hamiltonian $G$-variety.
\begin{itemize}
\item[(i)] If $(X,P)$ is symplectic, then $\nu^{-1}(\mathcal{S})$ is a symplectic subvariety of $X$. The symplectic structure on $\nu^{-1}(\mathcal{S})$ coincides with the Poisson structure that $\nu^{-1}(\mathcal{S})$ inherits as a Poisson transversal.
\item[(ii)] Suppose that $(X,P)$ is log symplectic with divisor $Y$, and that $Z$ is an irreducible component of $\nu^{-1}(\mathcal{S})$. The Poisson structure that $Z$ inherits from $\nu^{-1}(\mathcal{S})$ being a Poisson transversal is then log symplectic with divisor $Z\cap Y$.  
\end{itemize}
\end{prop}

\begin{rem}\label{Remark: Utility}
Analogous statements hold in the case of a Poisson variety $(X,P)$ endowed with a Hamiltonian ($G\times G$)-action and moment map $\nu:X\longrightarrow\g\oplus\g$. The preimage $\nu^{-1}(\mathcal{S}\times\mathcal{S})$ is a Poisson transversal in $X$, and Proposition \ref{Proposition: Utility} holds with ``$\nu^{-1}(\mathcal{S}\times\mathcal{S})$" in place of ``$\nu^{-1}(\mathcal{S})$".
\end{rem} 
 
\section{Poisson slices and the wonderful compactification}\label{Section: Poisson slices in the logarithmic cotangent bundle}
This section considers the implications of Proposition \ref{Proposition: Utility} for the cases $X=T^*G$ and  $X=T^*\overline{G}(\log D)$, where $T^*\overline{G}(\log D)$ is the log cotangent bundle of the wonderful compactification $\overline{G}$.

\subsection{The wonderful compactification of $G$}\label{Subsection: Wonderful}
Recall the Lie-theoretic notation and conventions established in Section \ref{Subsection: Lie-theoretic conventions}. The adjoint action of $G\times G$ on $\g\oplus \g$ is then given by $$(g_1,g_2)\cdot (x,y) = (\Ad_{g_1}(x),\Ad_{g_2}(y)),\quad (g_1,g_2)\in G\times G,\text{ }(x,y)\in\g\oplus\g.$$ This induces an action on the Grassmannian
$\mathrm{Gr}(n,\g\oplus\g)$ of $n$-dimensional subspaces in $\g\oplus\g$, where $n=\dim\g$. 
Now consider the point  
$$\g_\Delta := \{(x,x)\colon x\in \g\}\in\mathrm{Gr}(n,\g\oplus\g)$$ and set
$$\gamma_g:=(g,e)\cdot\mathfrak{g}_{\Delta}=\{(\mathrm{Ad}_g(x),x):x\in\mathfrak{g}\}\in\mathrm{Gr}(n,\g\oplus\g)$$ for each $g\in G$. 
The map
\begin{equation}\label{Equation: Def of gamma}\kappa:G\longrightarrow\mathrm{Gr}(n,\g\oplus \g),\quad g\mapsto\gamma_g,\quad g\in G\end{equation}
is a ($G\times G$)-equivariant locally closed embedding. The closure of its image is a ($G\times G$)-invariant closed subvariety of $\mathrm{Gr}(n,\g\oplus\g)$, often denoted $$\overline{G}\subseteq\mathrm{Gr}(n,\g\oplus \g)$$
and called the \textit{wonderful compactification} of $G$.

It is known that $\overline{G}$ is smooth \cite[Proposition 2.14]{Evens}, and that $D := \overline{G}\setminus G$ is a normal crossing divisor \cite[Theorem 2.22]{Evens}. One also knows that $\overline{G}$ is stratified into finitely many ($G\times G$)-orbits, indexed by the subsets $I\subseteq\Pi$ \cite[Theorem 2.22 and Remark 3.9]{Evens}. To describe the ($G\times G$)-orbit corresponding to $I\subseteq\Pi$, recall the Levi decompositions 
$$\p_I= \lf_I\oplus\uu_I\quad\text{and}\quad \p_I^-= \lf_I\oplus\uu_I^-$$
from Section \ref{Subsection: Lie-theoretic conventions}. These allow us to define an $n$-dimensional subspace of $\g\oplus\g$ by  
\begin{equation}\label{Equation: Basepoint} \p_I\times_{\lf_I}\p_I^- := \{(x,y)\in\p_I\oplus\p_I^-\colon x\text{ and }y\text{ have the same projection to }\lf_I\}.\end{equation} The ($G\times G$)-orbit corresponding to $I$ is then given by
$$(G\times G)(\p_I\times_{\lf_I}\p_I^-)\subseteq\overline{G}.$$
The following lemma (cf. \cite[Remark 3.9]{Evens}) implies that $\p_I\times_{\lf_I}\p_I^-\in\overline{G}$, justifying the inclusion asserted above. This lemma features prominently in later sections.
 
\begin{lem}\label{Lemma: 1PS}
Let $\tilde T\subseteq G$ be a maximal torus with Lie algebra $\tilde{\mathfrak{t}}\subseteq\g$, and consider a one-parameter subgroup $\lambda:\mathbb{C}^{\times}\longrightarrow \tilde T$. Let $\p$ be the parabolic subalgebra spanned by $\tilde{\mathfrak{t}}$ and the root spaces $\g_\alpha$ for all roots $\alpha$ of $(\g,\tilde{\mathfrak{t}})$ satisfying $(\alpha,\lambda)\geq 0$, where $(\cdot,\cdot)$ is the pairing between weights and coweights. Let $\lf\subseteq\g$ be the Levi subalgebra spanned by $\tilde{\mathfrak{t}}$ and all $\g_\alpha$ with $(\alpha,\lambda) = 0$. Write $\p^-$ for the opposite parabolic, spanned by $\tilde{\mathfrak{t}}$ and those root spaces $\g_\alpha$ such that $(\alpha,\lambda)\leq 0$. We then have 
$$\lim_{t\longrightarrow\infty}(\lambda(t),e)\cdot\g_\Delta = \p\times_{\lf}\p^-$$
in $\mathrm{Gr}(n,\g\oplus\g)$, where the right-hand side is defined analogously to \eqref{Equation: Basepoint}.
\end{lem}
\begin{proof}
Choose a Borel subalgebra $\tilde{\bb}\subseteq\g$ satisfying $\tilde{\mathfrak{t}}\subseteq\tilde{\bb}\subseteq \p$, and write $\tilde{\Phi}$, $\tilde{\Phi}^+$, and $\tilde{\Pi}$ for the associated sets of roots, positive roots, and simple roots, respectively. The subset $I:=\{\alpha\in\tilde{\Pi}:(\alpha,\lambda)=0\}$ then corresponds to the standard parabolic subalgebra $\p$. Let $\tilde{\Phi}_I^+\subseteq\tilde{\Phi}^+$ denote the set of positive roots in the $\mathbb{Z}$-span of $I$. We then have
$$\tilde{\Phi}_I^+ = \{\alpha\in\tilde{\Phi}^+:(\alpha,\lambda)=0\}\quad\text{and}\quad\tilde{\Phi}^+\setminus\tilde{\Phi}_I^+= \{\alpha\in\tilde{\Phi}^+:(\alpha,\lambda)>0\}.$$

Choose a non-zero root vector $e_\alpha\in\g_{\alpha}$ for each $\alpha\in\tilde{\Phi}$, and fix a basis $\{h_1,\dots,h_{\ell}\}$ of $\tilde{\mathfrak{t}}$. It follows that 
$$\{(e_\alpha,e_\alpha)\}_{\alpha\in\tilde{\Phi}}\cup\{(h_i,h_i)\}_{i=1}^{\ell}$$
is a basis of $\g_{\Delta}$.
Now set $$E_\alpha:=(e_\alpha,e_\alpha),\quad E^1_\alpha:= (e_\alpha,0),\quad\text{and}\quad E^2_\alpha:=(0,e_{\alpha})$$ 
for each $\alpha\in\tilde{\Phi}$, and also write $$H_i:=(h_i,h_i)$$ for $i=1,\dots,\ell$. Observe that $(\lambda(t),e)\cdot\g_\Delta$ then  has a basis of
$$\{t^{(\alpha,\lambda)}E^1_\alpha + E^2_\alpha\}_{\alpha\in\tilde\Phi^+\setminus\tilde{\Phi}_I^+}\cup\{t^{-(\alpha,\lambda)}E^1_{-\alpha}+E^2_{-\alpha}\}_{\alpha\in\tilde\Phi^+\setminus\tilde{\Phi}_I^+}\cup\{E_\beta\}_{\beta\in\tilde{\Phi}_I^+}\cup\{E_{-\beta}\}_{\beta\in\tilde{\Phi}_I^+}\cup\{H_i\}_{i=1}^\ell.$$ The image of $(\lambda(t),e)\cdot\g_\Delta$ under the Pl\"ucker embedding $$\vartheta:\mathrm{Gr}(n,\g\oplus\g)\longrightarrow \mathbb{P}(\Lambda^n(\g\oplus\g)),\quad V\mapsto [\Lambda^nV],\quad V\in\mathrm{Gr}(n,\g\oplus\g)$$ is therefore
$$\vartheta((\lambda(t),e)\cdot\g_\Delta) = \left[\bigwedge_{\alpha\in\tilde{\Phi}^+\setminus\tilde{\Phi}_I^+}(t^{(\alpha,\lambda)}E^1_\alpha+E^2_\alpha)\wedge \bigwedge_{i=1}^\ell H_i\wedge\bigwedge_{\beta\in\tilde{\Phi}_I^+}(E_{\beta}\wedge E_{-\beta})\wedge \bigwedge_{\alpha\in\tilde{\Phi}^+\setminus\tilde{\Phi}_I^+}(t^{-(\alpha,\lambda)}E^1_{-\alpha}+E^2_{-\alpha})\right].$$
We have 
$$\bigwedge_{\alpha\in\tilde\Phi^+\setminus\tilde{\Phi}_I^+}(t^{(\alpha,\lambda)}E^1_\alpha+E^2_\alpha) = t^{\sum_{\alpha\in\tilde\Phi^+\setminus\tilde{\Phi}_I^+}(\alpha,\lambda)}\bigg(\bigwedge_{\alpha\in\tilde{\Phi}^+\setminus\tilde{\Phi}_I^+}(E^1_\alpha+t^{-(\alpha,\lambda)}E^2_\alpha)\bigg),$$
and thus
$$\vartheta((\lambda(t),e)\cdot\g_\Delta) = \left[\bigwedge_{\alpha\in\tilde{\Phi}^+\setminus\tilde{\Phi}_I^+}(E^1_\alpha+t^{-(\alpha,\lambda)}E^2_\alpha)\wedge \bigwedge_{i=1}^\ell H_i\wedge\bigwedge_{\beta\in\tilde{\Phi}_I^+}(E_{\beta}\wedge E_{-\beta})\wedge \bigwedge_{\alpha\in\tilde{\Phi}^+\setminus\tilde{\Phi}_I^+}(t^{-(\alpha,\lambda)}E^1_{-\alpha}+E^2_{-\alpha})\right].$$
All exponents of $t$ appearing in this expression are strictly negative, implying that 
$$\lim_{t\longrightarrow\infty}\vartheta((\lambda(t),e)\cdot\g_\Delta) = \left[\bigwedge_{\alpha\in\tilde{\Phi}^+\setminus\tilde{\Phi}_I^+}E^1_\alpha\wedge \bigwedge_{i=1}^\ell H_i\wedge\bigwedge_{\beta\in\tilde{\Phi}_I^+}(E_{\beta}\wedge E_{-\beta})\wedge \bigwedge_{\alpha\in\tilde{\Phi}^+\setminus\tilde{\Phi}_I^+}E^2_{-\alpha}\right] = \vartheta(\p \times_{\lf}\p^-).$$
The desired conclusion now follows from the ($G\times G$)-equivariance of the Pl\"ucker embedding $\vartheta$.
\end{proof}

\subsection{The slices $G\times\mathcal{S}$ and $\overline{G\times\mathcal{S}}$}\label{Subsection: KW}
One can consider the \textit{log cotangent bundle}
$$\pi: T^*\overline{G}(\log D)\longrightarrow \overline{G},$$
i.e. the vector bundle associated with the locally free sheaf of logarithmic one-forms on $(\overline{G},D)$.
It turns out that $T^*\overline{G}(\log D)$ is the pullback of the tautological bundle $\mathbb
T\longrightarrow \mathrm{Gr}(n,\g\oplus \g)$ to the subvariety $\overline{G}$ (see \cite[Section 3.1]{BalibanuUniversal} or \cite[Example 2.5]{BrionVanishing}). In other words, 
\begin{equation}\label{Equation: T*GbarlogD}
T^*\overline{G}(\log D) = \{(\gamma,(x,y))\in \overline{G}\times(\g\oplus\g)\colon (x,y)\in\gamma\}
\end{equation}
and $$\pi(\gamma,(x,y))=\gamma$$ for all $(\gamma,(x,y))\in T^*\overline{G}(\log D)$. 

The variety $T^*\overline{G}(\log D)$ carries a natural log symplectic structure, some aspects of which we now describe. To this end, use the left trivialization and Killing form to identify $T^*G$ with $G\times\g$. It is straightforward to verify that
\begin{equation}\label{Equation: Action definition}(g_1,g_2)\cdot(h,x) = (g_1hg_2^{-1}, \Ad_{g_2}(x)), \quad (g_1,g_2)\in G\times G,\text{ }(h,x)\in G\times\g\end{equation}
defines a Hamiltonian action of $G\times G$ on $T^*G$, and that
\begin{equation}\label{Equation: MomentMapT*G}
\mu = (\mu_L,\mu_R): T^*G \longrightarrow \g\oplus\g,\quad (g,x)\mapsto (\Ad_g(x), x),\quad (g,x)\in G\times\g
\end{equation}
is a moment map.

The map
\begin{equation}\label{Equation: Open embedding}\Psi:T^*G\longrightarrow T^*\overline{G}(\log D),\quad (g,x)\mapsto (\gamma_g,(\mathrm{Ad}_g(x),x)),\quad (g,x)\in G\times\g\end{equation} 
defines an open embedding of $T^*G$ into $T^*\overline{G}(\log D)$, and it fits into a pullback square 
$$\begin{tikzcd}
T^*G \arrow[r, "\Psi"] \arrow[d]
& T^*\overline{G}(\log D) \arrow[d] \\
G \arrow[r, "\kappa"]
& \overline{G}
\end{tikzcd}.$$
This embedding is known to define a symplectomorphism from $T^*G$ to the unique open dense symplectic leaf in $T^*\overline{G}(\log D)$ \cite[Section 3.3]{BalibanuUniversal}. On the other hand, the following defines a Hamiltonian action of $G\times G$ on $T^*\overline{G}(\log D)$: 
\begin{equation}\label{Equation: Log action}(g_1,g_2)\cdot (\gamma,(x,y)):=((g_1,g_2)\cdot\gamma,(\mathrm{Ad}_{g_1}(x),\mathrm{Ad}_{g_2}(y))),\quad (g_1,g_2)\in G\times G,\text{ }(\gamma,(x,y))\in T^*\overline{G}(\log D),\end{equation}
where $(g_1,g_2)\cdot\gamma$ refers to the action of $G\times G$ on $\overline{G}$. An associated moment map is given by 
$$\overline{\mu}=(\overline{\mu}_L,\overline{\mu}_R):T^*\overline{G}(\log D)\longrightarrow\g\oplus\g,\quad (\gamma,(x,y))\mapsto (x,y),\quad (\gamma,(x,y))\in T^*\overline{G}(\log D)$$
(see \cite[Section 3.2]{BalibanuUniversal} and \cite[Example 2.5]{BrionVanishing}).
The open embedding $\Psi$ is then ($G\times G$)-equivariant and satisfies
$$\Psi^*\overline{\mu}=\mu.$$

We now discuss the Poisson slices
$$G\times\mathcal{S}=\mu_R^{-1}(\mathcal{S})\subseteq T^*G\quad\text{and}\quad\overline{G\times\mathcal{S}}:=\overline{\mu}_R^{-1}(\mathcal{S})\subseteq T^*\overline{G}(\log D),$$  
using \cite[Sections 3.2 and 3.2]{CrooksRoeserSlice} as our reference. The slice former is symplectic, while the latter is log symplectic. One also knows that $\Psi:T^*G\longrightarrow T^*\overline{G}(\log D)$ restricts to a $G$-equivariant open embedding
\begin{equation}\label{Equation: Restricted symplecto}\Psi\big\vert_{G\times\mathcal{S}}:G\times\mathcal{S}\longrightarrow\overline{G\times\mathcal{S}},\end{equation}
and that this embedding is a symplectomorphism onto the unique open dense symplectic leaf in $\overline{G\times\mathcal{S}}$. We also note that $G$ acts in a Hamiltonian fashion on $G\times\mathcal{S}$ (resp. $\overline{G\times\mathcal{S}}$)
via \eqref{Equation: Action definition} (resp. \eqref{Equation: Log action}) as the subgroup $G=G\times\{e\}\subseteq G\times G$. The maps $\mu_L:T^*G\longrightarrow\g$ and $\overline{\mu}_L:T^*\overline{G}(\log D)\longrightarrow\g$ then restrict to moment maps
$$\mu_{\mathcal{S}}:=\mu_L\bigg\vert_{G\times\mathcal{S}}:G\times\mathcal{S}\longrightarrow\g\quad\text{and}\quad\overline{\mu}_{\mathcal{S}}:=\overline{\mu}_L\bigg\vert_{\overline{G\times\mathcal{S}}}:\overline{G\times\mathcal{S}}\longrightarrow\g,$$ and we have a commutative diagram
\begin{equation}\begin{tikzcd}\label{Equation: Equivariant diagram}
G\times\mathcal{S} \arrow{rr}{\Psi\big\vert_{G\times\mathcal{S}}} \arrow[swap]{dr}{\mu_{\mathcal{S}}}& & \overline{G\times\mathcal{S}} \arrow{dl}{\overline{\mu}_{\mathcal{S}}}\\
& \mathfrak{g} & 
\end{tikzcd}.\end{equation}

An explicit description of $\overline{G\times\mathcal{S}}$ is obtained as follows. One begins by noting that $$\overline{G\times\mathcal{S}}=\overline{\mu}_R^{-1}(\mathcal S) = \{(\gamma,(x,y))\in\overline{G}\times(\g\oplus\g)\colon (x,y)\in \gamma\text{ and }y\in\mathcal S\}.$$ On the other hand, recall the adjoint quotient $$\chi:\g\longrightarrow\mathrm{Spec}(\mathbb{C}[\g]^G)$$ and the associated concepts discussed in Section \ref{Subsection: Lie-theoretic conventions}. The image of $\overline{\mu}$ is known to be
$$\{(x,y)\in\g\oplus\g\colon \chi(x) = \chi(y)\}$$
(see \cite[Proposition 3.4]{BalibanuUniversal}), and this implies the simplified description
\begin{equation}\label{Equation: mu2Stau}
\overline{G\times\mathcal{S}} = \{(\gamma,(x,x_{\mathcal{S}})):\gamma\in\overline{G},\text{ }x\in\g,\text{ and } (x,x_{\mathcal{S}})\in \gamma\}.
\end{equation}

\begin{rem}\label{Remark: Projective fibres}
Recall the bundle projection map $\pi:T^*\overline{G}(\log D)\longrightarrow\overline{G}$, and set
$$\pi_{\mathcal{S}}:=\pi\big\vert_{\overline{G\times\mathcal{S}}}:\overline{G\times\mathcal{S}}\longrightarrow\overline{G}.$$
The description \eqref{Equation: mu2Stau} implies that the product map
$$(\pi_{\mathcal{S}},\overline{\mu}_{\mathcal{S}}):\overline{G\times\mathcal{S}}\longrightarrow\overline{G}\times\g,\quad(\gamma,(x,x_{\mathcal{S}}))\mapsto (\gamma,x),\quad (\gamma,(x,x_{\mathcal{S}}))\in\overline{G\times\mathcal{S}}$$ is a closed embedding. We thereby obtain a commutative diagram
$$\begin{tikzcd}
\overline{G\times\mathcal{S}} \arrow{rr} \arrow[swap]{dr}{\overline{\mu}_{\mathcal{S}}}& & \overline{G}\times\g \arrow{dl}\\
& \mathfrak{g} &  
,\end{tikzcd}$$
where $\overline{G}\times\g\longrightarrow\g$ is projection to the second factor. We conclude that $\overline{\mu}_{\mathcal{S}}$ has projective fibres, so that \eqref{Equation: Equivariant diagram} realizes $\overline{\mu}_{\mathcal{S}}$ as a fibrewise compactification of $\mu_{\mathcal{S}}$. It also follows that 
$$\overline{\mu}_{\mathcal{S}}^{-1}(x)\longrightarrow\{\gamma\in\overline{G}:(x,x_{\mathcal{S}})\in\gamma\},\quad (\gamma,(x,x_{\mathcal{S}}))\mapsto\gamma,\quad (\gamma,(x,x_{\mathcal{S}}))\in\overline{\mu}_{\mathcal{S}}^{-1}(x)$$ is a variety isomorphism for each $x\in\g$.  
\end{rem}

\section{Relation to the standard family of Hessenberg varieties}\label{Section: Relation}
\subsection{The standard family of Hessenberg varieties}\label{Subsection: The standard family}
Recall the notation and conventions established in Section \ref{Subsection: Lie-theoretic conventions}, and let $B\subseteq G$ be the Borel subgroup with Lie algebra $\mathfrak{b}$. Suppose that $H\subseteq\g$ is a \textit{Hessenberg subspace}, i.e. a $B$-invariant subspace of $\g$ that contains $\mathfrak{b}$. Let $G\times B$ act on $G\times H$ via
$$(g,b)\cdot(h,x):=(ghb^{-1},\mathrm{Ad}_b(x)),\quad (g,b)\in G\times B,\text{ }(h,x)\in G\times H,$$ and consider the resulting smooth $G$-variety
$$G\times_BH:=(G\times H)/B.$$ 
Write $[g:x]$ for the equivalence class of $(g,x)\in G\times H$ in $G\times_B H$. The variety $G\times_B H$ is naturally Poisson, and its $G$-action is Hamiltonian with moment map
$$\nu_H:G\times_B H\longrightarrow\g,\quad [g:x]\mapsto\mathrm{Ad}_g(x),\quad [g:x]\in G\times_B H$$ (see \cite[ Section 4]{BalibanuUniversal}). Given any $x\in\g$, one writes
$$\mathrm{Hess}(x,H):=\nu_H^{-1}(x)$$ and calls this fibre the \textit{Hessenberg variety} associated to $H$ and $x$. The Poisson moment map $\nu_H$ is thereby called the \textit{family of Hessenberg varieties} associated to $H$.

\begin{rem}\label{Remark: Closed embedding}
Note that $G\times_B H$ is the total space of a $G$-equivariant vector bundle over $G/B$, and that the bundle projection map is
\begin{equation}\label{Equation: Bundle projection}\pi_H:G\times_BH\longrightarrow G/B,\quad [g:x]\mapsto[g],\quad [g:x]\in G\times_B H.\end{equation} 
The map $$(\pi_H,\nu_H):G\times_B H\longrightarrow G/B\times\g$$ is then a closed embedding. We also have a commutative diagram \begin{equation}\label{Equation: Hessenberg diagram}\begin{tikzcd}
G\times_B H \arrow{rr}{(\pi_H,\nu_H)} \arrow[swap]{dr}{\nu_H}& & G/B\times\g \arrow{dl}\\
& \mathfrak{g} &  
,\end{tikzcd}
\end{equation}
where $G/B\times\g\longrightarrow\g$ is projection to the second factor.
It follows that the fibres of $\nu_H$ are projective, and that $\pi_H$ restricts to a closed embedding $$\mathrm{Hess}(x,H)\hookrightarrow G/B$$ for each $x\in\g$. One may thereby regard Hessenberg varieties as closed subvarieties of $G/B$.
\end{rem}

In what follows, we restrict our attention to the so-called \textit{standard family} of Hessenberg varieties. This is defined to be the family $$\nu:=\nu_{\mathfrak{m}}:G\times_B\mathfrak{m}\longrightarrow\g$$ associated to the \textit{standard Hessenberg subspace}
\begin{equation}\label{Equation: StandardHessenbergSubspace}
\mathfrak{m}:=[\mathfrak{u},\mathfrak{u}]^{\perp}=\mathfrak{b}\oplus\bigoplus_{\alpha\in\Pi}\mathfrak{g}_{-\alpha}.
\end{equation}
 To study this family in more detail, we note the following consequences of the setup in Section \ref{Subsection: Lie-theoretic conventions}:
$$\xi=\sum_{\alpha\in\Pi}e_{-\alpha}\quad\text{and}\quad \g_{\eta}\subseteq\mathfrak{u},$$
where $$e_{-\alpha}\in\mathfrak{g}_{-\alpha}\setminus\{0\}$$ for all $\alpha\in\Pi$. These considerations imply that
$\mathcal{S}\subseteq\mathfrak{m}$, allowing one to define the map
$$\rho:G\times\mathcal{S}\longrightarrow G\times_B\mathfrak{m},\quad (g,x)\mapsto[g:x],\quad (g,x)\in G\times\mathcal{S}.$$ Let us also consider the $G$-invariant open dense subvariety
$$G\times_B\mathfrak{m}^{\times}\subseteq G\times_B\mathfrak{m},$$ where $\mathfrak{m}^{\times}\subseteq\mathfrak{m}$ is the $B$-invariant open subvariety defined by
$$\mathfrak{m}^{\times}:=\mathfrak{b}+\sum_{\alpha\in\Pi}(\mathfrak{g}_{-\alpha}\setminus\{0\}):=\bigg\{x+\sum_{\alpha\in\Pi}c_{\alpha}e_{-\alpha}:x\in\mathfrak{b}\text{ and }c_{\alpha}\in\mathbb{C}^{\times}\text{ for all }\alpha\in\Pi\bigg\}.$$ One then has the following consequence of \cite[Theorem 41]{AbeCrooks} and \cite[Section 4.2 and Theorem 4.16]{BalibanuUniversal}.

\begin{prop}\label{Proposition: Slice log symplectic}
The Poisson variety $G\times_B\mathfrak{m}$ is log symplectic with $G\times_B\mathfrak{m}^{\times}$ as its unique open dense symplectic leaf. The map $\rho$ is a $G$-invariant symplectomorphism onto the leaf $G\times_B\mathfrak{m}^{\times}$.  
\end{prop}

\subsection{Some toric geometry} 
This section uses the techniques of toric geometry to compare the fibres  $\nu^{-1}(x)=\mathrm{Hess}(x,\mathfrak{m})$ and $\overline{\mu}_{\mathcal{S}}^{-1}(x)$ over regular semisimple elements $x\in\mathfrak{g}^{\text{rs}}$. Many of the underlying ideas appear in \cite{BalibanuUniversal} and \cite{Evens}. We therefore do not regard this section as containing any original material.    

We begin by observing that $G_x$ acts on the fibres of $\nu$ and $\overline{\mu}_{\mathcal{S}}$ over $x$ for all $x\in\g$. This leads to the following two lemmas, parts of which are well-known. To this end, recall the notation and discussion from Section \ref{Subsection: Lie-theoretic conventions}. 

\begin{lem}\label{Lemma: Unique open dense}
Suppose that $x\in\mathfrak{g}^{\emph{r}}$. The following statements hold.
\begin{itemize}
\item[(i)] There is a unique open dense orbit of $G_x$ in $\mathrm{Hess}(x,\mathfrak{m})$.
\item[(ii)] The group $G_x$ acts freely on the above-mentioned orbit.
\item[(iii)] If $g\in G$ satisfies $x=\mathrm{Ad}_g(x_{\mathcal{S}})$, then $[g:x_{\mathcal{S}}]$ belongs to the above-mentioned $G_x$-orbit.
\item[(iv)] If $x\in\mathfrak{g}^{\emph{rs}}$, then $\mathrm{Hess}(x,\mathfrak{m})$ is a smooth, projective, toric $G_x$-variety. 
\end{itemize}
\end{lem} 
 
\begin{proof}
Let $g\in G$ be such that $x=\mathrm{Ad}_g(x_{\mathcal{S}})$. Our first observation is that $$\nu([g:x_{\mathcal{S}}])=\mathrm{Ad}_{g}(x_{\mathcal{S}})=x,$$ i.e. $[g:x_{\mathcal{S}}]\in\mathrm{Hess}(x,\mathfrak{m})$. At the same time, Corollaries 3 and 14 in \cite{PrecupTransformation} imply that $\mathrm{Hess}(x,\mathfrak{m})$ is irreducible and $\ell$-dimensional. Claims (i), (ii), and (iii) would therefore follow from our showing the $G_x$-stabilizer of $[g:x_{\mathcal{S}}]$ to be trivial.

Suppose that $h\in G_x$ is such that $[hg:x_{\mathcal{S}}]=[g:x_{\mathcal{S}}]$. It follows that $(hgb^{-1},\mathrm{Ad}_b(x_{\mathcal{S}}))=(g,x_{\mathcal{S}})$ for some $b\in B$. Since the $B$-stabilizer of every point in $\mathcal{S}$ is trivial (see \cite[Lemma 4.9]{BalibanuUniversal}), we must have $b=e$. This yields the identity $hg=g$, or equivalently $h=e$. In light of the conclusion reached in the previous paragraph, Claims (i), (ii), and (iii) hold.

Claim (iv) is well-known and follows from Theorems 6 and 11 in \cite{DeMari}.
\end{proof}

Let us also recall the notation $\gamma_g$ defined in Section \ref{Subsection: Wonderful}. 

\begin{lem}\label{Lemma: Slice lemma}
Suppose that $x\in\mathfrak{g}^{\emph{r}}$. The following statements hold.
\begin{itemize}
\item[(i)] There is a unique open dense orbit of $G_x$ in $\overline{\mu}_{\mathcal{S}}^{-1}(x)$.
\item[(ii)] The group $G_x$ acts freely on the above-mentioned orbit.
\item[(iii)] If $g\in G$ satisfies $x=\mathrm{Ad}_g(x_{\mathcal{S}})$, then $(\gamma_g,(x,x_{\mathcal{S}}))$ belongs to the above-mentioned $G_x$-orbit.
\item[(iv)] If $x\in\mathfrak{g}^{\emph{rs}}$, then $\overline{\mu}_{\mathcal{S}}^{-1}(x)$ is a smooth, projective, toric $G_x$-variety.  
\end{itemize}
\end{lem}

\begin{proof}
The moment map $\overline{\mu}_{\mathcal{S}}$ is $G$-equivariant, so that acting by the element $g^{-1}$ defines a variety isomorphism
\begin{equation}\label{Equation: Prelim iso}\overline{\mu}_{\mathcal{S}}^{-1}(x)\overset{\cong}\longrightarrow\overline{\mu}_{\mathcal{S}}^{-1}(x_{\mathcal{S}}).
\end{equation} 
Since the latter variety is given by
\begin{equation}\label{Equation: Mild}\overline{\mu}_{\mathcal{S}}^{-1}(x_{\mathcal{S}})=\{(\gamma,(x_{\mathcal{S}},x_{\mathcal{S}})):\gamma\in\overline{G}\text{ and }(x_{\mathcal{S}},x_{\mathcal{S}})\in\gamma\},
\end{equation}
we can use \cite[Corollary 3.12]{BalibanuUniversal} and conclude that $\overline{\mu}_{\mathcal{S}}^{-1}(x_{\mathcal{S}})$ is irreducible and $\ell$-dimensional. It follows that $\overline{\mu}_{\mathcal{S}}^{-1}(x)$ is irreducible and $\ell$-dimensional. As with the proof of our previous lemma, Claims (i), (ii), and (ii) would now follow from knowing $(\gamma_g,(x,x_{\mathcal{S}}))$ to have a trivial $G_x$-stabilizer. This is established via a straightforward calculation.

To verify Claim (iv), recall the isomorphism \eqref{Equation: Prelim iso}. This isomorphism implies that $\overline{\mu}_{\mathcal{S}}^{-1}(x)$ is a smooth, projective, toric $G_x$-variety if and only if $\overline{\mu}_{\mathcal{S}}^{-1}(x_{\mathcal{S}})$ is a smooth, projective, toric $g^{-1}G_xg=G_{x_{\mathcal{S}}}$-variety. The latter condition holds because of \eqref{Equation: Mild}, \cite[Corollary 3.12]{BalibanuUniversal}, and the fact that the closure of $G_{x_{\mathcal{S}}}$ in $\overline{G}$ is a smooth, projective, toric $G_{x_{\mathcal{S}}}$-variety \cite[Remark 4.5]{Evens}. Our proof is therefore complete.
\end{proof}

Our next two lemmas study the $G_x$-fixed point sets
$$\mathrm{Hess}(x,\mathfrak{m})^{G_x}\subseteq\mathrm{Hess}(x,\mathfrak{m})\quad\text{and}\quad\overline{\mu}_{\mathcal{S}}^{-1}(x)^{G_x}\subseteq\overline{\mu}_{\mathcal{S}}^{-1}(x)$$
for each $x\in\mathfrak{g}^{\text{rs}}$. To formulate these results, let $\mathcal{B}$ denote the flag variety of all Borel subalgebras in $\g$ and set
$$\mathcal{B}_x:=\{\widetilde{\mathfrak{b}}\in\mathcal{B}:x\in \widetilde{\mathfrak{b}}\}$$ for each $x\in\g$. Recall that \begin{equation}\label{Equation: Iso}G/B\longrightarrow\mathcal{B},\quad [g]\mapsto\mathrm{Ad}_g(\mathfrak{b}),\quad [g]\in G/B\end{equation} defines a $G$-equivariant variety isomorphism. Let us also consider the map
$$\pi_{\mathfrak{m}}:G\times_B\mathfrak{m}\longrightarrow G/B,\quad [g:x]\mapsto [g],\quad [g:x]\in G\times_B\mathfrak{m}$$ and its composition with \eqref{Equation: Iso}, i.e.
$$\varpi_{\mathfrak{m}}:G\times_B\mathfrak{m}\longrightarrow\mathcal{B},\quad [g:x]\mapsto\mathrm{Ad}_g(\mathfrak{b}),\quad [g:x]\in G\times_B\mathfrak{m}.$$

\begin{lem}\label{Lem: First fixed}
If $x\in\mathfrak{g}^{\emph{rs}}$, then there is a canonical bijection 
$$\mathcal{B}_x\longrightarrow\mathrm{Hess}(x,\mathfrak{m})^{G_x},\quad \widetilde{\mathfrak{b}}\mapsto z(\widetilde{\mathfrak{b}}),\quad \widetilde{\mathfrak{b}}\in\mathcal{B}_x$$
satisfying
$$\widetilde{\mathfrak{b}}=\varpi_{\mathfrak{m}}(z(\widetilde{\mathfrak{b}}))$$
for all $\widetilde{\mathfrak{b}}\in \mathcal{B}_x$.
\end{lem}

\begin{proof}
Recall that $\pi_{\mathfrak{m}}$ restricts to a closed embedding
$$\mathrm{Hess}(x,\mathfrak{m})\longrightarrow G/B$$ (see Remark \ref{Remark: Closed embedding}).
Composing this embedding with \eqref{Equation: Iso}, we deduce that $\varpi_{\mathfrak{m}}$ restricts to a closed embedding
\begin{equation}\label{Equation: Preinjection}
\mathrm{Hess}(x,\mathfrak{m})\longrightarrow\mathcal{B}.
\end{equation} This embedding is $G_x$-equivariant, implying that it restricts to an injection
\begin{equation}\label{Equation: Injection}
\mathrm{Hess}(x,\mathfrak{m})^{G_x}\longrightarrow\mathcal{B}^{G_x}=\mathcal{B}_x.
\end{equation}

We claim that \eqref{Equation: Injection} is also surjective. To this end, suppose that $\widetilde{\mathfrak{b}}\in\mathcal{B}_x$. We then have $\widetilde{\mathfrak{b}}=\mathrm{Ad}_g(\mathfrak{b})$ for some $g\in G$. It follows that $x=\mathrm{Ad}_g(y)$ for some $y\in\mathfrak{b}$, so that we have a point $[g:y]\in\mathrm{Hess}(x,\mathfrak{m})$. The image of this point under \eqref{Equation: Preinjection} is $\widetilde{\mathfrak{b}}$. Noting that $\widetilde{\mathfrak{b}}\in\mathcal{B}_x=\mathcal{B}^{G_x}$ and that \eqref{Equation: Preinjection} is injective and $G_x$-equivariant, we conclude that $[g:y]\in\mathrm{Hess}(x,\mathfrak{m})^{G_x}$. This establishes that \eqref{Equation: Injection} is surjective, i.e. that it is a bijection. The bijection advertised in the statement of the lemma is obtained by inverting \eqref{Equation: Injection}.
\end{proof}

If $x\in\mathfrak{g}^{\text{rs}}$, then the Cartan subalgebra $\mathfrak{g}_x$ satisfies $\mathfrak{g}_x\subseteq\widetilde{\mathfrak{b}}$ for all $\widetilde{\mathfrak{b}}\in\mathcal{B}_x$ (see \cite[Lemma 3.1.4]{Chriss}). This fact gives rise to an opposite Borel subalgebra $\widetilde{\mathfrak{b}}^{-}\in\mathcal{B}_x$ for all $\widetilde{\mathfrak{b}}\in\mathcal{B}_x$. We may therefore define 
$$\theta(\widetilde{\mathfrak{b}}):=\widetilde{\mathfrak{b}}\times_{\mathfrak{g}_x}\widetilde{\mathfrak{b}}^-$$ analogously to \eqref{Equation: Basepoint}.
\begin{lem}\label{Lemma: Fixed points}
Suppose that $x\in\mathfrak{g}^{\emph{rs}}$ and let $g\in G$ be such that $x=\mathrm{Ad}_g(x_{\mathcal{S}})$. We then have a bijection defined by
$$\mathcal{B}_x\longrightarrow\overline{\mu}_{\mathcal{S}}^{-1}(x)^{G_x},\quad \widetilde{\mathfrak{b}}\mapsto \bigg((e,g^{-1})\cdot\theta(\widetilde{\mathfrak{b}}),(x,x_{\mathcal{S}})\bigg),\quad \widetilde{\mathfrak{b}}\in \mathcal{B}_x.$$
\end{lem}

\begin{proof}
Note that conjugation by $g^{-1}$ defines a bijection
$\phi:\mathcal{B}_x\longrightarrow\mathcal{B}_{x_{\mathcal{S}}}$.
At the same time, consider the automorphism of $\overline{G\times\mathcal{S}}$ through which the element $g^{-1}$ acts. This automorphism restricts to a bijection
$$\overline{\mu}_{\mathcal{S}}^{-1}(x)^{G_x}\overset{\cong}\longrightarrow\overline{\mu}_{\mathcal{S}}^{-1}(x_{\mathcal{S}})^{G_{x_{\mathcal{S}}}}.$$ It also sends 
$$\bigg((e,g^{-1})\cdot\theta(\widetilde{\mathfrak{b}}),(x,x_{\mathcal{S}})\bigg)\mapsto (\theta(\phi(\widetilde{\mathfrak{b}})),(x_{\mathcal{S}},x_{\mathcal{S}}))$$ for all $\widetilde{\mathfrak{b}}\in\mathcal{B}_x$, where the subspaces $\{\theta(\widetilde{\mathfrak{b}})\}_{\widetilde{\mathfrak{b}}\in\mathcal{B}_{x_{\mathcal{S}}}}$ are defined analogously to the $\{\theta(\widetilde{\mathfrak{b}})\}_{\widetilde{\mathfrak{b}}\in\mathcal{B}_x}$. It therefore suffices to prove that 
$$\mathcal{B}_{x_{\mathcal{S}}}\longrightarrow\overline{\mu}_{\mathcal{S}}^{-1}(x_{\mathcal{S}})^{G_{x_{\mathcal{S}}}},\quad \widetilde{\mathfrak{b}}\mapsto(\theta(\widetilde{\mathfrak{b}}),(x_{\mathcal{S}},x_{\mathcal{S}})),\quad \widetilde{\mathfrak{b}}\in\mathcal{B}_{x_{\mathcal{S}}}$$ defines a bijection. In other words, it suffices to prove our lemma under the assumption that $x\in\mathcal{S}\cap\mathfrak{g}^{\text{rs}}$ and $g=e$.

Let us make the assumption indicated above. The fibre $\overline{\mu}_{\mathcal{S}}^{-1}(x)$ is then given by
$$\overline{\mu}_{\mathcal{S}}^{-1}(x)=\{(\gamma,(x,x)):\gamma\in\overline{G}\text{ and }(x,x)\in\gamma\}.$$ It now follows from \cite[Corollary 3.12]{BalibanuUniversal} that
$$\overline{G_x}\longrightarrow\overline{\mu}_{\mathcal{S}}^{-1}(x),\quad\gamma\mapsto (\gamma,(x,x)),\quad \gamma\in\overline{G_x}$$ defines an isomorphism of varieties, where $\overline{G_x}$ denotes the closure of $G_x$ in $\overline{G}$. This isomorphism is $G_x$-equivariant if one lets $G_x$ act on $\overline{G_x}$ by restricting the ($G\times G$)-action on $\overline{G}$ to $G_x=G_x\times\{e\}\subseteq G\times G$. We therefore have 
\begin{equation}\label{Equation: Pre}\overline{\mu}_{\mathcal{S}}^{-1}(x)^{G_x}=\{(\gamma,(x,x)):\gamma\in(\overline{G_x})^{G_x}\}.\end{equation} At the same time, we will prove that
\begin{equation}\label{Equation: Pre2}(\overline{G_x})^{G_x}=\{\theta(\widetilde{\mathfrak{b}}):\widetilde{\mathfrak{b}}\in\mathcal{B}_x\}\end{equation}
in Lemma \ref{Lemma: PoissonSlicefixedpoints}.
Our current lemma now follows from \eqref{Equation: Pre}, \eqref{Equation: Pre2}, and the conclusion of the previous paragraph.
\end{proof}

\begin{lem}\label{Lemma: PoissonSlicefixedpoints}
We have 
$$(\overline{G_x})^{G_x}=\{\theta(\widetilde{\mathfrak{b}}):\widetilde{\mathfrak{b}}\in\mathcal{B}_x\}$$
for all $x\in\mathfrak{g}^{\emph{rs}}$, where $G_x$ acts on $\overline{G_x}$ as the subgroup $G_x=G_x\times\{e\}\subseteq G\times G$ via the ($G\times G$)-action on $\overline{G}$.
\end{lem}
\begin{proof}
It suffices to assume that $x\in\mathfrak{t}^{\text{r}}$, so that $G_x = T$. Now suppose that $\widetilde{\mathfrak{b}}\in\mathcal{B}_x$. Lemma \ref{Lemma: 1PS} implies that $$\theta(\widetilde\bb) = \lim_{t\longrightarrow\infty}(\lambda(t),e)\cdot\g_\Delta$$ for a suitable one-parameter subgroup $\lambda:\C^{\times}\longrightarrow T$, while we observe that $$(\lambda(t),e)\cdot\g_{\Delta}\in\kappa(T)\subseteq\overline{G}$$ for all $t\in\mathbb{C}^{\times}$. It follows that 
$\theta(\widetilde{\mathfrak{b}})\in\overline{T}$. A direct calculation establishes that $\theta(\widetilde{\mathfrak{b}})$ is a $T$-fixed point, yielding the inclusion 
$$\{\theta(\widetilde{\mathfrak{b}}):\widetilde{\mathfrak{b}}\in\mathcal{B}_x\}\subseteq\overline{T}^T.$$

To establish the opposite inclusion, suppose that $\gamma \in (\overline{T})^{T}$. Using the ($G\times G$)-orbit decomposition of $\overline{G}$ from Section \ref{Subsection: Wonderful}, we may find $I\subseteq \Pi$ and $g_1,g_2\in G$ such that 
$$\gamma = (g_1,g_2)\cdot \p_I \times_{\lf_I} \p_I^-.$$
We will first show that $g_1$ can be taken to lie in $N_G(T)$ without the loss of generality. Note that $\widetilde{\mathfrak{b}}:=\mathrm{Ad}_{g_1}(\mathfrak{b})$ will then be a Borel subalgebra containing $x$. We will then explain that $\gamma = \theta(\widetilde\bb)$, completing the proof. 
In what follows, we will need the following description of the ($G\times G$)-stabilizer of $\p_I\times_{\lf_I}\p_I^-\in\overline{G}$:
\begin{equation}\label{Equation: StabilisermI}
(G\times G)_{I} := \{(l_1u,l_2v)\in L_IU_I\times L_IU_I^-\colon l_1l_2^{-1}\in Z(L_I)\},
\end{equation}
where $L_I\subseteq G$ is the Levi subgroup with Lie algebra $\lf_I$, $Z(L_I)$ is the centre of $L_I$, and $U_I\subseteq G$ (resp. $U_I^{-}\subseteq G$) is the unipotent subgroup with Lie algebra $\uu_I$ (resp. $\uu_I^{-}$) \cite[Proposition 2.25]{Evens}.

Since $\gamma$ is fixed by $T$, we have
$$(t,e)\cdot\gamma =\gamma$$ 
for all $t\in T$. It follows that
$$(tg_1,g_2)\cdot \p_I \times_{\lf_I} \p_I^- = (g_1,g_2)\cdot\p_I \times_{\lf_I} \p_I^-,$$
i.e. 
$$(g_1^{-1}tg_1,e)\cdot \p_I \times_{\lf_I} \p_I^- = \p_I \times_{\lf_I} \p_I^-.$$
We deduce that $(g_1^{-1}tg_1,e)\in (G \times G)_{I}$, which by \eqref{Equation: StabilisermI} implies that 
$$g_1^{-1}tg_1 \in U_IZ(L_I)$$  
for all $t\in T$. In other words, 
$$g_1^{-1}Tg_1 \subseteq U_I Z(L_I)\subseteq P_I,$$
where $P_I\subseteq G$ is the parabolic subgroup with Lie algebra $\mathfrak{p}_I$. 
Noting that $g_1^{-1}Tg_1$ and $T$ are maximal tori in $P_I$, we can find a $p\in P_I$ such that 
$$g_1^{-1}Tg_1 = pTp^{-1}.$$
This shows that $(g_1p)^{-1}Tg_1p = T$, so that $g_1p\in N_G(T)$. Now use the decomposition $P_I = L_IU_I$ to write $p = lu$ with $l\in L_I$ and $u\in U_I$. Equation \eqref{Equation: StabilisermI} then tells us that $(p,l)\in (G\times G)_I$, yielding 
$$\gamma = (g_1,g_2)\cdot \p_I\times_{\lf_I}\p_I^- = (g_1p,g_2l)\cdot\p_I\times_{\lf_I}\p_I^-.$$
We may therefore take $g_1\in N_G(T)$ without the loss of generality. 

Now we show that $I=\emptyset$. An argument given above establishes that 
$$T = g_1^{-1}Tg_1 \subseteq U_IZ(L_I),$$
i.e. $T\subseteq U_IZ(L_I)$. Let $t\in T$ be arbitrary and write $t=vz$ with $v\in U_I$ and $z\in Z(L_I)$. Since $T$ contains $Z(L_I)$, we have $z\in T$ and $v = tz^{-1}\in T\cap U_I =\{e\}$. We conclude that $v = e$ and $t\in Z(L_I)$. This shows that $T\subseteq Z(L_I)$, which can only happen if $T = Z(L_I)$. One deduces that $I=\emptyset$, and this yields $\mathfrak{p}_I=\mathfrak{b}$, $\mathfrak{p}_I^{-}=\mathfrak{b}^{-}$, $\mathfrak{l}_I=\mathfrak{t}$, and
$$\gamma = (g_1,g_2)\cdot\bb\times_{\mathfrak{t}}\bb^-.$$ 

We now prove that $g_2 = g_1\in N_G(T)$ without the loss of generality. Our first observation is that  
$$(\Ad_{g_1^{-1}}(x),\Ad_{g_2^{-1}}(x)) \in \bb\times_{\mathfrak{t}}\bb^-,$$ as $(x,x)\in\gamma$. Since $g_1\in N_G(T)$, we have $\Ad_{g_1^{-1}}(x)\in \mathfrak{t}$. These last two sentences then force  
$$\Ad_{g_1^{-1}}(x) - \Ad_{g_2^{-1}}(x)\in\uu^-$$ to hold.
We also note that $\Ad_{g_1}^{-1}(x)\in\mathfrak{t}$ is regular, which by \cite[Lemma 3.1.44]{Chriss} implies that $\Ad_{g_1}^{-1}(x)+ \uu^-$ is a $U^-$-orbit. We can therefore find $u_-\in U^-$ such that 
$$\Ad_{g_2^{-1}}(x) = \Ad_{u_-g_1^{-1}}(x),$$
i.e. 
$(g_2u_-)g_1^{-1}\in G_x = T$.
This shows that $g_2u_-\in Tg_1 = g_1T$, where we have used the fact that $g_1\in N_G(T)$. We may therefore write  $g_2u_- = g_1t$ for some $t\in T$, i.e. 
$$g_1=g_2u_-t^{-1}.$$ 
Equation \eqref{Equation: StabilisermI} also implies that $(e,u_-t^{-1}) \in (G\times G)_\emptyset$, giving $(e,u_-t^{-1})\cdot\bb\times_{\mathfrak{t}}\bb^- =\bb\times_{\mathfrak{t}}\bb^-$. This in turn implies that 
$$\gamma = (g_1,g_2)\cdot\bb\times_{\mathfrak{t}}\bb^-= (g_1,g_2u_-t^{-1})\cdot \bb\times_{\mathfrak{t}}\bb^- = (g_1,g_1)\cdot \bb\times_{\mathfrak{t}}\bb^-=\theta(\widetilde{\mathfrak{b}}),$$ where $\widetilde{\mathfrak{b}}=\mathrm{Ad}_{g_1}(\mathfrak{b})$. Our proof is therefore complete.
\end{proof}

The preceding results have implications for the toric geometries of $\mathrm{Hess}(x,\mathfrak{m})$ and $\overline{\mu}_{\mathcal{S}}^{-1}(x)$, where $x\in\g^{\text{rs}}$. The following elementary lemma is needed to realize these implications.

\begin{lem}\label{Lemma: BB}
Suppose that $S$ is a complex torus, and that $X$ is a smooth, projective, toric $S$-variety with finitely many $S$-fixed points. Let $\lambda:\mathbb{C}^{\times}\longrightarrow S$ be a coweight of $S$, and assume that $(\alpha,\lambda)\neq 0$ for all $y\in X^S$ and all weights $\alpha$ of the isotropy $S$-representation $T_yX$. Assume that $x$ belongs to the unique open dense orbit $S$-orbit in $X$.
We have $$\lim_{t\longrightarrow\infty}\lambda(t)\cdot x=z$$
for some $z\in X^S$ satisfying $(\alpha,\lambda)<0$ for all $S$-weights $\alpha$ of $T_zX$.  
\end{lem}

\begin{proof}
Let $\mathbb{C}^{\times}$ act on $X$ through $\lambda$. If $y\in X^S$, then $T_yY$ is an isotropy representation of both $S$ and $\mathbb{C}^{\times}$. Each weight of the $\mathbb{C}^{\times}$-representation is given by $(\alpha,\lambda)$ for a suitable weight $\alpha$ of the $S$-representation. Since $(\alpha,\lambda)\neq 0$ for all $y\in X^S$ and $S$-weights $\alpha$ of $T_yX$, the previous sentence implies that $X^S=X^{\mathbb{C}^{\times}}$. 

Now note that $\lim_{t\longrightarrow\infty}\lambda(t)\cdot x$ exists and coincides with a point $z\in X^{\mathbb{C}^{\times}}=X^S$ (e.g. by \cite[Lemma 2.4.1]{Chriss}). The previous paragraph explains that each $\mathbb{C}^{\times}$-weight of $T_z X$ takes the form $(\alpha,\lambda)$, where $\alpha$ is an $S$-weight of $T_z X$. General facts about Bia{\l}ynicki-Birula decompositions (e.g. \cite[Theorem 2.4.3]{Chriss}) now imply that $(\alpha,\lambda)<0$ for all $S$-weights $\alpha$ of $T_zX$.      	
\end{proof}

Suppose that $x\in\g^{\text{rs}}$. Recall that a coweight $\lambda:\mathbb{C}^{\times}\longrightarrow G_x$ is called \textit{regular} if $(\alpha,\lambda)\neq 0$ for all roots $\alpha$ of $(\mathfrak{g},\g_x)$. Let us also recall the notation adopted in Lemmas \ref{Lem: First fixed} and \ref{Lemma: Fixed points}.

\begin{lem}\label{Lemma: Limit}
Suppose that $x\in\mathfrak{g}^{\emph{rs}}$ and let $g\in G$ be such that $x=\mathrm{Ad}_g(x_{\mathcal{S}})$. Given any regular coweight $\lambda:\mathbb{C}^{\times}\longrightarrow G_x$ and element $\widetilde{\mathfrak{b}}\in\mathcal{B}_x$, one has the following equivalence:
$$\lim_{t\longrightarrow\infty}\lambda(t)\cdot[g:x_{\mathcal{S}}]=z(\widetilde{\mathfrak{b}})\Longleftrightarrow\lim_{t\longrightarrow\infty}\lambda(t)\cdot (\gamma_g,(x,x_{\mathcal{S}}))=\bigg((e,g^{-1})\cdot\theta(\widetilde{\mathfrak{b}}),(x,x_{\mathcal{S}})\bigg),$$
where the left- (resp. right-) hand side is computed in $\mathrm{Hess}(x,\mathfrak{m})$ (resp. $\overline{\mu}_{\mathcal{S}}^{-1}(x)$).
\end{lem}

\begin{proof}
Let us write $\Pi(\widetilde{\mathfrak{b}})\subseteq\mathfrak{g}_x^*$ for the set of simple roots determined by $\mathfrak{g}_x$ and $\widetilde{\mathfrak{b}}\in\mathcal{B}_x$. By \cite[Lemma 7]{DeMari}, the $G_x$-weights of the isotropy representation $T_{z(\widetilde{\mathfrak{b}})}\mathrm{Hess}(x,\mathfrak{m})$ form the set $-\Pi(\widetilde{\mathfrak{b}})$. This has two implications for any regular coweight $\lambda:\mathbb{C}^{\times}\longrightarrow G_x$. One is that $(\alpha,\lambda)\neq 0$ for all $\widetilde{\mathfrak{b}}\in\mathcal{B}_x$ and $G_x$-weights $\alpha$ of $T_{z(\widetilde{\mathfrak{b}})}\mathrm{Hess}(x,\mathfrak{m})$. The second implication is the existence of a unique  $\widetilde{\mathfrak{b}}\in\mathcal{B}_x$ such that $(\alpha,\lambda)<0$ for all $G_x$-weights $\alpha$ of $T_{z(\widetilde{\mathfrak{b}})}\mathrm{Hess}(x,\mathfrak{m})$. These last few sentences and Lemma \ref{Lemma: BB} imply the following about a given regular coweight $\lambda$ and element $\widetilde{\mathfrak{b}}\in\mathcal{B}_x$:
$$\lim_{t\longrightarrow\infty}\lambda(t)\cdot[g:x_{\mathcal{S}}]=z(\widetilde{\mathfrak{b}})\Longleftrightarrow(\alpha,\lambda)<0\text{ for all }\alpha\in -\Pi(\widetilde{\mathfrak{b}}),$$ or equivalently
$$\lim_{t\longrightarrow\infty}\lambda(t)\cdot[g:x_{\mathcal{S}}]=z(\widetilde{\mathfrak{b}})\Longleftrightarrow(\alpha,\lambda)>0\text{ for all }\alpha\in\Pi(\widetilde{\mathfrak{b}}).$$ 

In light of the previous paragraph, we are reduced to proving that
\begin{equation}\label{Equation: Implications}\lim_{t\longrightarrow\infty}\lambda(t)\cdot (\gamma_g,(x,x_{\mathcal{S}}))=\bigg((e,g^{-1})\cdot\theta(\widetilde{\mathfrak{b}}),(x,x_{\mathcal{S}})\bigg)\Longleftrightarrow(\alpha,\lambda)>0\text{ for all }\alpha\in\Pi(\widetilde{\mathfrak{b}}).\end{equation} Consider the action of $\lambda(t)$ indicated above, noting that it coincides with the action of $(\lambda(t),e)\in G\times G$ on points in $\overline{\mu}_{\mathcal{S}}^{-1}(x)\subseteq T^*\overline{G}(\log D)$. At the same time, the actions of $(\lambda(t),e)$ and $(e,g)$ on $T^*\overline{G}(\log D)$ commute with one another for all $t\in\mathbb{C}^{\times}$. Acting by $(e,g)$ therefore allows us to reformulate \eqref{Equation: Implications} as
$$\lim_{t\longrightarrow\infty}\lambda(t)\cdot (\mathfrak{g}_{\Delta},(x,x))=(\theta(\widetilde{\mathfrak{b}}),(x,x))\Longleftrightarrow(\alpha,\lambda)>0\text{ for all }\alpha\in \Pi(\widetilde{\mathfrak{b}}),$$ or equivalently
$$\lim_{t\longrightarrow\infty}(\lambda(t),e)\cdot\mathfrak{g}_{\Delta}=\theta(\widetilde{\mathfrak{b}})\Longleftrightarrow(\alpha,\lambda)>0\text{ for all }\alpha\in\Pi(\widetilde{\mathfrak{b}}).$$ This last equivalence is a straightforward consequence of Lemma \ref{Lemma: 1PS}, and our proof is complete.
\end{proof}

\subsection{Some results on gluing}
Let us use \eqref{Equation: Restricted symplecto} to identify $G\times\mathcal{S}$ with the unique open dense symplectic leaf in $\overline{G\times\mathcal{S}}$. The $G$-equivariant map $\rho:G\times\mathcal{S}\longrightarrow G\times_B\mathfrak{m}$ in Section \ref{Subsection: The standard family}
is then given by
$$\rho(\gamma_g,(\mathrm{Ad}_g(x),x))=[g:x]$$ for all $(\gamma_g,(\mathrm{Ad}_g(x),x))\in G\times\mathcal{S}\subseteq\overline{G\times\mathcal{S}}$. One also has the commutative diagram
\begin{equation}\label{Equation: Newest diagram}\begin{tikzcd}
G\times\mathcal{S} \arrow{rr}{\rho} \arrow[swap]{dr}{\overline{\mu}_{\mathcal{S}}\big\vert_{G\times\mathcal{S}}}& & G\times_B\mathfrak{m} \arrow{dl}{\nu}\\
& \mathfrak{g} & . 
\end{tikzcd}
\end{equation} 

Now fix $x\in\g^{\text{rs}}$ and consider the fibres
$$(\overline{\mu}_{\mathcal{S}}\big\vert_{G\times\mathcal{S}})^{-1}(x)=\overline{\mu}_{\mathcal{S}}^{-1}(x)\cap (G\times\mathcal{S})=:\overline{\mu}_{\mathcal{S}}^{-1}(x)^{\circ}$$
and
$$\nu^{-1}(x)=\mathrm{Hess}(x,\mathfrak{m}).$$ The commutative diagram \eqref{Equation: Newest diagram} implies that $\rho$ restricts to a $G_x$-equivariant morphism
$$\rho_x:\overline{\mu}_{\mathcal{S}}^{-1}(x)^{\circ}\longrightarrow\mathrm{Hess}(x,\mathfrak{m}).$$ 

\begin{lem}\label{Lemma: IsoMomentFibres}
If $x\in\mathfrak{g}^{\emph{rs}}$, then there exists a unique $G_x$-equivariant variety isomorphism
$$\overline{\rho}_x:\overline{\mu}_{\mathcal{S}}^{-1}(x)\overset{\cong}\longrightarrow\mathrm{Hess}(x,\mathfrak{m})$$ that extends $\rho_x$.	
\end{lem}

\begin{proof}
Choose $g\in G$ such that $x=\mathrm{Ad}_g(x_{\mathcal{S}})$. Lemma \ref{Lemma: Slice lemma} combines with our description of $G\times\mathcal{S}$ as a subset of $\overline{G\times\mathcal{S}}$ to imply that $\overline{\mu}_{\mathcal{S}}^{-1}(x)^{\circ}$ is the unique open dense $G_x$-orbit in $\overline{\mu}_{\mathcal{S}}^{-1}(x)$.  
The uniqueness of $\overline{\rho}_x$ is then a consequence of $\overline{\mu}_{\mathcal{S}}^{-1}(x)^{\circ}$ being dense in $\overline{\mu}_{\mathcal{S}}^{-1}(x)$. 

To establish existence, recall that $\overline{\mu}_{\mathcal{S}}^{-1}(x)$ and $\mathrm{Hess}(x,\mathfrak{m})$ are smooth, projective, toric $G_x$-varieties with respective points \begin{equation}\label{Equation: Points} (\gamma_g,(x,x_{\mathcal{S}}))\quad\text{and}\quad [g:x_{\mathcal{S}}]\end{equation}
in their open dense $G_x$-orbits (see Lemmas \ref{Lemma: Unique open dense} and \ref{Lemma: Slice lemma}). Recall also that elements of $\overline{\mu}_{\mathcal{S}}^{-1}(x)^{G_x}$ and $\mathrm{Hess}(x,\mathfrak{m})^{G_x}$ are in correspondence with elements of $\mathcal{B}_x$ (see Lemmas \ref{Lem: First fixed} and \ref{Lemma: Fixed points}), and that the points \eqref{Equation: Points} limit to corresponding $G_x$-fixed points under a given regular coweight (see Lemma \ref{Lemma: Limit}). By these last two sentences, there exists a unique $G_x$-variety isomorphism
$$\overline{\rho}_x:\overline{\mu}_{\mathcal{S}}^{-1}(x)\overset{\cong}\longrightarrow\mathrm{Hess}(x,\mathfrak{m})$$
satisfying
$$\overline{\rho}_x(\gamma_g,(x,x_{\mathcal{S}}))=[g:x_{\mathcal{S}}].$$ Note that $\rho_x$ is also $G_x$-equivariant and satisfies
$$\rho_x(\gamma_g,(x,x_{\mathcal{S}}))=[g:x_{\mathcal{S}}].$$ Since $\overline{\rho}_x$ and $\rho_x$ are $G_x$-equivariant, these last two equations imply that
$\overline{\rho}_x$ and $\rho_x$ coincide on $$G_x\cdot (\gamma_g,(x,x_{\mathcal{S}}))= \overline{\mu}_{\mathcal{S}}^{-1}(x)^{\circ}.$$ 
\end{proof}

\begin{rem}\label{Remark: Continuous}
One may take the following more holistic view of the proof given above. Each of $\overline{\mu}_{\mathcal{S}}^{-1}(\mathfrak{g}^{\text{rs}})\longrightarrow\mathfrak{g}^{\text{rs}}$ and $\nu^{-1}(\mathfrak{g}^{\text{rs}})\longrightarrow\mathfrak{g}^{\text{rs}}$ is a family of smooth projective toric varieties. In each family, the fan of the toric variety over $x\in\mathfrak{g}^{\text{rs}}$ is the decomposition of $\mathfrak{g}_x$ into Weyl chambers (see Lemma \ref{Lemma: Limit} and \eqref{Equation: Implications}). The isomorphisms $\overline{\rho}_x$ simply result from this fibrewise equality of fans. One then sees that the $\overline{\rho}_x$ glue together to define a variety isomorphism
$$\overline{\mu}_{\mathcal{S}}^{-1}(\g^{\text{rs}})\overset{\cong}\longrightarrow\nu^{-1}(\g^{\text{rs}}).$$ Since $\overline{\rho}_x$ is defined on $\overline{\mu}_{\mathcal{S}}^{-1}(x)$ and takes values in $\mathrm{Hess}(x,\mathfrak{m})=\nu^{-1}(x)$ for each $x\in\mathfrak{g}^{\text{rs}}$, it follows that \begin{equation}\label{Equation: Newer Newest diagram}\begin{tikzcd}
\overline{\mu}_{\mathcal{S}}^{-1}(\g^{\text{rs}}) \arrow{rr}{\cong} \arrow[swap]{dr}{\overline{\mu}_{\mathcal{S}}\big\vert_{\overline{\mu}_{\mathcal{S}}^{-1}(\g^{\text{rs}})}}& & \nu^{-1}(\g^{\text{rs}}) \arrow{dl}{\nu\big\vert_{\nu^{-1}(\g^{\text{rs}})}}\\
& \mathfrak{g} & . 
\end{tikzcd}
\end{equation} commutes.
\end{rem}

\subsection{The isomorphism in codimension one}
We now undertake a brief digression on transverse intersections in a $G$-variety. To this end, let $X$ be an arbitrary algebraic variety. Write $X_{\text{sing}}$ and $X_{\text{smooth}}:=X\setminus X_{\text{sing}}$ for the singular and smooth loci of $X$, respectively.
Given any closed subvariety $Y\subseteq X$, let
$$\mathrm{codim}_X(Y):=\dim X-\dim Y$$ denote the codimension of $Y$ in $X$.
 
\begin{lem}\label{Lemma: Transverse equivariant}
Suppose that $X$ is a smooth $G$-variety containing a $G$-invariant closed subvariety $Y\subseteq X$. Let $Z\subseteq X$ be a smooth closed subvariety with the property of being transverse to each $G$-orbit in $X$. If $W$ is an irreducible component of $Y\cap Z$, then 
$$\mathrm{codim}_{Z}(W)\geq\mathrm{codim}_X(Y).$$ 
\end{lem}

\begin{proof}
Set $Y_0:=Y$ and recursively define $Y_{j+1}:=(Y_j)_{\text{sing}}$ for all $j\in\mathbb{Z}_{\geq 0}$. One then has a descending chain
$$Y=Y_0\supseteq Y_1\supseteq Y_2\supseteq\cdots$$ of $G$-invariant closed subvarieties in $Y$. The locally closed subvarieties $(Y_j)_{\mathrm{smooth}}\subseteq Y$ are also $G$-invariant, and we observe that these subvarieties have a disjoint union equal to $Y$.

Let $k$ be the smallest non-negative integer for which $(Y_k)_{\text{smooth}}\cap W\neq\emptyset$. It is then straightforward to verify that $(Y_k)_{\text{smooth}}\cap W=W\setminus Y_{k+1}$. We conclude that $(Y_k)_{\text{smooth}}\cap W$ is an open dense subset of $W$, and this implies that $W\subseteq Y_k$.

Now choose a point $w\in (Y_k)_{\text{smooth}}\cap W$. Since $(Y_k)_{\text{smooth}}$ is $G$-invariant and $Z$ is transverse to every $G$-orbit in $X$, the varieties $(Y_k)_{\text{smooth}}$ and $Z$ have a transverse intersection at $w$. It follows that $w$ is a smooth point of $Y_k\cap Z$, and that
$$\dim W\leq \dim(T_w(Y_k\cap Z))=\dim(T_wY_k)+\dim Z-\dim X\leq \dim Y+\dim Z-\dim X.$$ This yields the conclusion
$$\mathrm{codim}_{Z}(W)\geq\mathrm{codim}_X(Y).$$ 
\end{proof}

We will ultimately apply this result in a context relevant to our main results. To this end, use \eqref{Equation: Open embedding} to identify $T^*G$ with the unique open dense symplectic leaf in $T^*\overline{G}(\log D)$. Consider the open subvariety
$$T^*\overline{G}(\log D)^{\circ}:=T^*G\cup\overline{\mu}_L^{-1}(\g^{\text{rs}})\subseteq T^*\overline{G}(\log D)$$ and its complement
$$T^*\overline{G}(\log D)':=T^*\overline{G}(\log D)\setminus T^*\overline{G}(\log D)^{\circ},$$
where $\overline{\mu}_L:T^*\overline{G}(\log D)\longrightarrow\g$ is defined in Section \ref{Subsection: KW}.
\begin{lem}\label{Lemma: First codimension}
We have $$\mathrm{codim}_{T^*\overline{G}(\log D)}(T^*\overline{G}(\log D)')\geq 2.$$
\end{lem}

\begin{proof}
We begin by observing that
$$T^*\overline{G}(\log D)'=\pi^{-1}(D)\setminus\overline{\mu}_L^{-1}(\g^{\text{rs}}),$$
where $\pi:T^*\overline{G}(\log D)\longrightarrow\overline{G}$ is the bundle projection and $D=\overline{G}\setminus G$. It therefore suffices to prove that $\overline{\mu}_L^{-1}(\g^{\text{rs}})$ meets each irreducible component of $\pi^{-1}(D)$. 

Now note that 
$$D=\bigcup_{I}\overline{(G\times G)\p_I\times_{\lf_I}\p_I^-}$$
is the decomposition of $D$ into irreducible components, where $I$ ranges over all subsets of the form $\Pi\setminus\{\alpha\}$, $\alpha\in\Pi$ \cite[Section 3.1]{BalibanuUniversal}. We conclude that $$\pi^{-1}(D) = \bigcup_I\pi^{-1}\bigg(\overline{(G\times G)\p_I\times_{\lf_I}\p_I^-}\bigg)$$ is the decomposition of $\pi^{-1}(D)$ into irreducible components. It therefore suffices to prove that $\mu^{-1}(\g^{\text{rs}})$ meets $\pi^{-1}((G\times G)\p_I\times_{\lf_I}\p_I^-)$ for all $I\subseteq\Pi$ of the form $I=\Pi\setminus\{\alpha\}$, $\alpha\in\Pi$.

Choose $x\in\mathfrak{t}^{\text{r}}$ and an element $g\in G$ satisfying $x=\mathrm{Ad}_g(x_{\mathcal{S}})$. We then have $$(x,x_{\mathcal{S}})\in (e,g^{-1})\cdot\p_I\times_{\lf_I}\p_I^-$$ for all $I\subseteq\Pi$. We also observe that
$$\bigg((e,g^{-1})\cdot\p_I\times_{\lf_I}\p_I^-,(x,x_{\mathcal{S}})\bigg)\in\overline{\mu}_L^{-1}(\g^{\text{rs}})\cap \pi^{-1}((G\times G)\p_I\times_{\lf_I}\p_I^-)$$ for all $I\subseteq\Pi$. This completes the proof.  
\end{proof}

Now consider the open subvariety
$$\overline{G\times\mathcal{S}}^{\circ}:=(G\times\mathcal{S})\cup\overline{\mu}_{\mathcal{S}}^{-1}(\g^{\text{rs}})\subseteq\overline{G\times\mathcal{S}}$$ and its complement
$$\overline{G\times\mathcal{S}}':=\overline{G\times\mathcal{S}}\setminus \overline{G\times\mathcal{S}}^{\circ}.$$ 

\begin{lem}\label{Lemma: Codimgeq2}
	We have $$\mathrm{codim}_{\overline{G\times\mathcal{S}}}(\overline{G\times\mathcal{S}}')\geq 2.$$
\end{lem}

\begin{proof}
Our task is to prove that each irreducible component of $\overline{G\times\mathcal{S}}'$ has codimension at least two in $\overline{G\times\mathcal{S}}$. We begin by showing ourselves to be in the situation of Lemma \ref{Lemma: Transverse equivariant}. To this end, consider $X:=T^*\overline{G}(\log D)$ and the action of $G=\{e\}\times G\subseteq G\times G$ on $X$. Let $Y$ be the subvariety $T^*\overline{G}(\log D)'\subseteq T^*\overline{G}(\log D)$ considered in Lemma \ref{Lemma: First codimension}, and set $Z:=\overline{G\times\mathcal{S}}$. Since $Z=\overline{\mu}_R^{-1}(\mathcal{S})$, \cite[Proposition 3.6]{CrooksRoeserSlice} implies that $Z$ is transverse to every $G$-orbit in $X$. We also note that 
$$Y\cap Z=\overline{G\times\mathcal{S}}\setminus (T^*G\cup\overline{\mu}_L^{-1}(\g^{\text{rs}}))=\overline{G\times\mathcal{S}}\setminus((G\times\mathcal{S})\cup\overline{\mu}_{\mathcal{S}}^{-1}(\g^{\text{rs}}))=\overline{G\times\mathcal{S}}',$$
and that 
$$\mathrm{codim}_X(Y)=\mathrm{codim}_{T^*\overline{G}(\log D)}(T^*\overline{G}(\log D)')\geq 2$$
by Lemma \ref{Lemma: First codimension}. The desired conclusion now follows immediately from Lemma \ref{Lemma: Transverse equivariant}.
\end{proof}

Now recall the notation and content of Section \ref{Subsection: The standard family}. Consider the open subvariety $$(G\times_B\mathfrak{m})^{\circ}:=(G\times_B\mathfrak{m}^{\times})\cup\nu^{-1}(\g^{\text{rs}})\subseteq G\times_B\mathfrak{m}$$
and its complement
$$(G\times_B\mathfrak{m})':=G\times_B\mathfrak{m}\setminus (G\times_B\mathfrak{m})^{\circ}.$$

\begin{lem}\label{Lemma: Next extension}
We have
$$\mathrm{codim}_{G\times_B\mathfrak{m}}((G\times_B\mathfrak{m})')\geq 2.$$
\end{lem}
 
\begin{proof}
It suffices to prove that the open subvariety $\nu^{-1}(\g^{\text{rs}})\subseteq G\times_B\mathfrak{m}$ meets each irreducible component of $(G\times_B\mathfrak{m})\setminus (G\times_B\mathfrak{m}^{\times})$. To this end, suppose that $\beta\in\Pi$ is a simple root and define
$$\mathfrak{m}_{\beta}^{\times}:=\mathfrak{b}+\sum_{\alpha\in\Pi\setminus\{\beta\}}(\g_{-\alpha}\setminus\{0\}):=\bigg\{x+\sum_{\alpha\in\Pi\setminus\{\beta\}}c_{\alpha}e_{-\alpha}:x\in\mathfrak{b}\text{ and }c_{\alpha}\in\mathbb{C}^{\times}\text{ for all }\alpha\in\Pi\setminus\{\beta\}\bigg\}.$$ The irreducible components of $(G\times_B\mathfrak{m})\setminus (G\times_B\mathfrak{m}^{\times})$ are then the subvarieties
$$G\times_B\mathfrak{m}_{\beta}^{\times}\subseteq G\times_B\mathfrak{m},\quad\beta\in\Pi.$$ 

Fix $\beta\in\Pi$ and choose an element $x\in\mathfrak{t}^{\text{r}}$. By \cite[Lemma 3.1.44]{Chriss}, the elements
$$x\quad\text{and}\quad x_{\beta}:=x+\sum_{\alpha\in\Pi\setminus\{\beta\}}e_{-\alpha}$$ are in the same adjoint $G$-orbit. It follows that $x_{\beta}\in\g^{\text{rs}}$, while we also observe that $x_{\beta}\in\mathfrak{m}_{\beta}^{\times}$. These considerations imply that $$[e:x_{\beta}]\in (G\times_B\mathfrak{m}_{\beta}^{\times})\cap\nu^{-1}(\g^{\text{rs}}),$$ forcing
$$(G\times_B\mathfrak{m}_{\beta}^{\times})\cap\nu^{-1}(\g^{\text{rs}})\neq\emptyset$$ to hold. In light of the previous paragraph, our proof is complete.
\end{proof} 

We also require the following elementary lemma in order to prove the main result of this section.

\begin{lem}\label{Lemma: Extended gluing isomorphism}
Suppose that $X$ and $Y$ are irreducible varieties. Let $U_1,U_2\subseteq X$ and $V_1,V_2\subseteq Y$ be open subsets, and assume that $U_1$ and $U_2$ are non-empty. Suppose that $f_1:U_1\longrightarrow V_1$ and $f_2:U_2\longrightarrow V_2$ are variety isomorphisms that agree on points in $U_1\cap U_2$, and let $f:U_1\cup U_2\longrightarrow V_1\cup V_2$ be the result of gluing $f_1$ and $f_2$ along $U_1\cap U_2$. Then $f$ is a variety ismomorphism.
\end{lem}

\begin{proof}
Our hypotheses imply that $U_1\cap U_2$ is a non-empty open subset of $X$. It follows that $f_1(U_1\cap U_2)$ is a non-empty open subset of $V_1$. One also has $f_1(U_1\cap U_2)\subseteq V_2$, as $f_1$ and $f_2$ coincide on $U_1\cap U_2$. We conclude that $f_1(U_1\cap U_2)$ is a non-empty open subset of $V_1\cap V_2$.

Now suppose that $y\in f_1(U_1\cap U_2)$. It follows that $$f_1(x)=y=f_2(x)$$ for some $x\in U_1\cap U_2$, so that $$f_1^{-1}(y)=x=f_2^{-1}(y).$$ This establishes that $f_1^{-1}$ and $f_2^{-1}$ coincide on $f_1(U_1\cap U_2)$. The conclusion of the previous paragraph and the irreducibility of $V_1\cap V_2$ now imply that $f_1^{-1}$ and $f_2^{-1}$ agree on $V_1\cap V_2$. We may therefore glue $f_1^{-1}$ and $f_2^{-1}$ along $V_1\cap V_2$ to obtain a morphism $V_1\cup V_2\longrightarrow U_1\cup U_2$. This morphism is easily checked to be an inverse of $f$. Our proof is therefore complete.  
\end{proof}  

The main result of this section is as follows.

\begin{thm}\label{Theorem: First main}
There exists a unique Hamiltonian $G$-variety isomorphism
$$\overline{\rho}^{\circ}:\overline{G\times\mathcal{S}}^{\circ}\overset{\cong}\longrightarrow (G\times_B\mathfrak{m})^{\circ}$$ that extends $\rho:G\times\mathcal{S}\longrightarrow G\times_B\mathfrak{m}^{\times}$.
\end{thm}

\begin{proof}
Uniqueness is an immediate consequence of $G\times\mathcal{S}$ being dense in $\overline{G\times\mathcal{S}}^{\circ}$. To establish existence, recall the variety isomorphism
\begin{equation}\label{Equation: Continuous bijection}\overline{\mu}_{\mathcal{S}}^{-1}(\g^{\text{rs}})\overset{\cong}\longrightarrow\nu^{-1}(\g^{\text{rs}})\end{equation} discussed in Remark \ref{Remark: Continuous}. Lemma \ref{Lemma: IsoMomentFibres} implies that this isomorphism coincides with $\rho$ on the overlap $(G\times\mathcal{S})\cap\overline{\mu}_{\mathcal{S}}^{-1}(\g^{\text{rs}})$. We may therefore apply Lemma \ref{Lemma: Extended gluing isomorphism} to \eqref{Equation: Continuous bijection} and $\rho$, where the latter is regarded as an isomorphism onto its image $\rho(G\times\mathcal{S})=G\times_B\mathfrak{m}^{\times}$. It follows that $\rho$ extends to a variety isomorphism
$$\overline{G\times \mathcal S}^{\circ} = (G\times\mathcal{S})\cup\overline{\mu}_{\mathcal{S}}^{-1}(\g^{\text{rs}})\overset{\overline{\rho}^{\circ}}\longrightarrow (G\times_B\mathfrak{m}^{\times})\cup\nu^{-1}(\mathfrak{g}^{\text{rs}})=(G\times_B\mathfrak{m})^{\circ}.$$ The commutative diagrams \eqref{Equation: Newest diagram} and \eqref{Equation: Newer Newest diagram} imply that $\overline{\rho}^{\circ}$ intertwines the restrictions of $\overline{\mu}_{\mathcal{S}}$ and $\nu$ to $\overline{G\times\mathcal{S}}^{\circ}$ and $(G\times_B\mathfrak{m})^{\circ}$, respectively. 

It remains to establish that $\overline{\rho}^{\circ}$ is $G$-equivariant and Poisson. The former is a straightforward consequence of $\rho$ being $G$-equivariant and $G\times\mathcal{S}$ being dense in $\overline{G\times\mathcal{S}}$. The latter follows from the fact that $\rho$ defines a symplectomorphism between the open dense symplectic leaves $G\times\mathcal{S}\subseteq\overline{G\times\mathcal{S}}$ and $G\times_B\mathfrak{m}^{\times}\subseteq G\times_B\mathfrak{m}$ (see Proposition \ref{Proposition: Slice log symplectic}). Our proof is therefore complete.
\end{proof}

\begin{rem}
In light of Lemmas \ref{Lemma: Codimgeq2} and \ref{Lemma: Next extension}, one has the following coarser reformulation of Theorem \ref{Theorem: First main}: $\rho$ extends to an isomorphism $\overline{G\times\mathcal{S}}\overset{\cong}\longrightarrow G\times_B\mathfrak{m}$ in codimension one. This observation features prominently in proof of Theorem \ref{Theorem: GxBHKostantWhittaker}.
\end{rem}

\begin{rem}
Consider the restricted moment maps
$$\overline{\mu}_{\mathcal{S}}^{\circ}:=\overline{\mu}_{\mathcal{S}}\bigg\vert_{\overline{G\times\mathcal{S}}^{\circ}}\quad\text{and}\quad \nu^{\circ}:=\nu\bigg\vert_{(G\times_B\mathfrak{m})^{\circ}},$$ as well as the Hamiltonian $G$-variety isomorphism $\overline{\rho}^{\circ}:\overline{G\times\mathcal{S}}^{\circ}\longrightarrow (G\times_B\mathfrak{m})^{\circ}$. Theorem
\ref{Theorem: First main} tells us that 
\begin{equation}\label{Equation: Prelim diag}\begin{tikzcd}
\overline{G\times\mathcal{S}}^{\circ} \arrow{rr}{\overline{\rho}^{\circ}} \arrow[swap]{dr}{\overline{\mu}_{\mathcal{S}}^{\circ}}& & (G\times_B\mathfrak{m})^{\circ} \arrow{dl}{\nu^{\circ}}\\
& \mathfrak{g} &  
\end{tikzcd}\end{equation}
commutes, and pulling back along the inclusion $\mathcal{S}\hookrightarrow\g$ yields a commutative diagram
$$\begin{tikzcd}
\overline{G\times\mathcal{S}}^{\circ}\cap\overline{\mu}_{\mathcal{S}}^{-1}(\mathcal{S})=(\overline{\mu}_{\mathcal{S}}^{\circ})^{-1}(\mathcal{S}) \arrow{rr}{} \arrow[start anchor={[xshift=-1.0ex]}, shorten >=6mm,
end anchor={[yshift=-2.5ex]north east}]{dr}{}& & (\nu^{\circ})^{-1}(\mathcal{S})=(G\times_B\mathfrak{m})^{\circ}\cap\nu^{-1}(\mathcal{S}) \arrow[start anchor={[xshift=-5.0ex]}, shorten >=6mm,
end anchor={[yshift=0.9ex]south west}]{dl}{}\\
& \mathcal{S} & 
\end{tikzcd}.$$
The horizontal arrow is a Poisson variety isomorphism between the Poisson slices $(\overline{\mu}_{\mathcal{S}}^{\circ})^{-1}(\mathcal{S})\subseteq\overline{G\times\mathcal{S}}^{\circ}$ and $(\nu^{\circ})^{-1}(\mathcal{S})\subseteq (G\times_B\mathfrak{m})^{\circ}$, as follows from $\overline{\rho}^{\circ}$ being a Poisson variety isomorphism. On the other hand,
\cite[Example 5.5]{CrooksRoeserSlice} explains that $\overline{\mu}_{\mathcal{S}}^{-1}(\mathcal{S})\longrightarrow\mathcal{S}$ is Balib\u{a}nu's fibrewise compactified universal centralizer $\overline{\mathcal{Z}_{\g}}\longrightarrow\mathcal{S}$. Our pullback diagram thereby becomes
$$\begin{tikzcd}
\overline{G\times\mathcal{S}}^{\circ}\cap\overline{\mathcal{Z}_{\g}} \arrow{rr}{\cong} \arrow[swap]{dr}{}& & (G\times_B\mathfrak{m})^{\circ}\cap\nu^{-1}(\mathcal{S}) \arrow{dl}{}\\
& \mathcal{S} & 
\end{tikzcd}.$$
The horizontal arrow is precisely the restriction of Balib\u{a}nu's Poisson isomorphism 
$$\overline{\mathcal{Z}_{\g}}\overset{\cong}\longrightarrow\nu^{-1}(\mathcal{S})$$
to $\overline{G\times\mathcal{S}}^{\circ}\cap\overline{\mathcal{Z}_{\g}}$, as follows from comparing the proofs of Theorem \ref{Theorem: First main} and \cite[Proposition 4.8]{BalibanuUniversal}.   
\end{rem}

\subsection{The bimeromorphism}
Let us briefly recall Remmert's notion of a meromorphic map \cite{Remmert}, as well as several related concepts. To this end, suppose that $X$ and $Y$ are complex manifolds. We also suppose that $f:X\longrightarrow Y$ is a set-theoretic relation, specified by a subset
$$\mathrm{graph}(f)\subseteq X\times Y.$$ One then calls $f$ a \textit{meromorphic map} if the following conditions are satisfied:
\begin{itemize}
\item[(i)] $\mathrm{graph}(f)$ is an analytic subset of $X\times Y$;
\item[(ii)] there exists an open dense subset $Z\subseteq X$ for which $f\big\vert_Z:Z\longrightarrow Y$ is a holomorphic map and $$\mathrm{graph}(f)=\overline{\mathrm{graph}(f\big\vert_{Z})}$$ in $X\times Y$;
\item[(iii)] the projection $X\times Y\longrightarrow X$ restricts to a proper map $\mathrm{graph}(f)\longrightarrow X$. 
\end{itemize}

Suppose that $X$ and $Y$ come equipped with holomorphic $G$-actions. We then refer to a meromorphic map $f:X\longrightarrow Y$ as being $G$-equivariant if $\mathrm{graph}(f)$ is invariant under the diagonal $G$-action on $X\times Y$.
 
Let $f:X\longrightarrow Y$ and $h:Y\longrightarrow Z$ be meromorphic maps. Declare $(f,h)$ to be composable if there exists an open dense subset $W\subseteq Y$ such that $h\big\vert_{W}:W\longrightarrow Z$ is a holomorphic map and $$f^{-1}(W):=\{x\in X:(x,y)\in\mathrm{graph}(f)\text{ for some }y\in W\}$$ is dense in $X$. The composite relation $$h\circ f:X\longrightarrow Z$$ is then a meromorphic map. One calls $f$ a \textit{bimeromorphism} if there exists a meromorphic map $\ell:Y\longrightarrow X$ such that $(f,\ell)$ and $(\ell,f)$ are composable and $$\ell\circ f:X\longrightarrow X\quad\text{and}\quad f\circ\ell:Y\longrightarrow Y$$ are the identity relations.   

We may now formulate and prove the main result of this section. 

\begin{thm}\label{Theorem: GxBHKostantWhittaker}
There exists a unique bimeromorphism
$$\overline{\rho}:\overline{G\times\mathcal S}\overset{\cong}\longrightarrow G\times_B\mathfrak{m}$$ that extends $\overline{\rho}^{\circ}$. This bimeromorphism is $G$-equivariant and makes
\begin{equation}\label{Equation: Newer diagram}\begin{tikzcd}
\overline{G\times\mathcal{S}} \arrow{rr}{\overline{\rho}} \arrow[swap]{dr}{\overline{\mu}_{\mathcal{S}}}& & G\times_B\mathfrak{m} \arrow{dl}{\nu}\\
& \mathfrak{g} &  
\end{tikzcd}
\end{equation}
commute.
\end{thm}
\begin{proof}
Uniqueness is an immediate consequence of $G\times\mathcal{S}$ being dense in $\overline{G\times\mathcal{S}}$. 

Now recall the $G$-equivariant closed embedding
\begin{equation}\label{Equation: Nice closed embedding}(\pi,\nu):G\times_B\mathfrak{m}\longrightarrow G/B\times\g\end{equation} defined in Remark \ref{Remark: Closed embedding}, and form the $G$-equivariant composite map
$$\epsilon:=(\pi\circ\overline{\rho}^{\circ},\nu\circ\overline{\rho}^{\circ}):\overline{G\times\mathcal{S}}^{\circ}\longrightarrow G/B\times\g.$$
Lemma \ref{Lemma: Codimgeq2} and \cite[Theorem 3.42]{McKay} imply that the first component $\pi\circ\overline{\rho}^{\circ}$ extends to a meromorphic map $$\widetilde{\pi\circ\overline{\rho}^{\circ}}:\overline{G\times\mathcal{S}}\longrightarrow G/B.$$ On the other hand, \eqref{Equation: Newest diagram} and \eqref{Equation: Newer Newest diagram} tell us that $\overline{\mu}_{\mathcal{S}}$ extends the second component $\nu\circ\overline{\rho}^{\circ}$ to $\overline{G\times\mathcal{S}}$.

Let us consider the product meromorphic map
$$\widetilde{\epsilon}:=(\widetilde{\pi\circ\overline{\rho}^{\circ}},\overline{\mu}_{\mathcal{S}}):\overline{G\times\mathcal{S}}\longrightarrow G/B\times\mathfrak{g},$$ i.e. the meromorphic map with graph
$$\mathrm{graph}(\widetilde{\epsilon})=\{(\alpha,([g],x))\in (\overline{G\times\mathcal{S}})\times (G/B\times\g):(\alpha,[g])\in\mathrm{graph}(\widetilde{\pi\circ\overline{\rho}^{\circ}})\text{ and }\overline{\mu}_{\mathcal{S}}(\alpha)=x\}.$$ The last two sentences of the preceding paragraph imply that $\widetilde{\epsilon}$ extends $\epsilon$. This amounts to having an inclusion
$$\mathrm{graph}(\epsilon)\subseteq\mathrm{graph}(\widetilde{\epsilon}).$$ The irreducibility of $\overline{G\times\mathcal{S}}$ (see \cite[Theorem 3.14]{CrooksRoeserSlice}) and \cite[Remark 2.3]{Biliotti} also force
$\mathrm{graph}(\widetilde{\epsilon})$ to be irreducible, and we note that
$$\dim(\mathrm{graph}(\epsilon))=\dim(\overline{G\times\mathcal{S}}^{\circ})=\dim(\overline{G\times\mathcal{S}})=\dim(\mathrm{graph}(\widetilde{\epsilon})).$$ It follows that
$$\mathrm{graph}(\widetilde{\epsilon})=\overline{\mathrm{graph}(\epsilon)},$$where the closure is taken in $(\overline{G\times\mathcal{S}})\times (G/B\times\g)$. This combines with the $G$-equivariance of $\epsilon$ to imply that $\widetilde{\epsilon}$ is $G$-equivariant. Another consequence is that
\begin{equation}\label{Equation: Nice inclusion}\mathrm{pr}(\mathrm{graph}(\widetilde{\epsilon}))\subseteq\overline{\mathrm{pr}(\mathrm{graph}(\epsilon))},\end{equation}
where $$\mathrm{pr}:(\overline{G\times\mathcal{S}})\times (G/B\times\g)\longrightarrow G/B\times\g$$ is the obvious projection.
We also know that $\mathrm{pr}(\mathrm{graph}(\epsilon))$ is the image of $\epsilon$, and that the latter contains the image of $\rho(G\times\mathcal{S})$ under the closed embedding \eqref{Equation: Nice closed embedding}. Since $\rho(G\times\mathcal{S})$ is dense in $G\times_B\mathfrak{m}$, this implies that $\overline{\mathrm{pr}(\mathrm{graph}(\epsilon))}$ is the image of \eqref{Equation: Nice closed embedding}. The inclusion \eqref{Equation: Nice inclusion} thus forces $\mathrm{pr}(\mathrm{graph}(\widetilde{\epsilon}))$ to be contained in the image of \eqref{Equation: Nice closed embedding}, onto which \eqref{Equation: Nice closed embedding} is a $G$-equivariant isomorphism. One may thereby regard $\widetilde{\epsilon}$ as a $G$-equivariant meromorphic map
$$\overline{\rho}:\overline{G\times\mathcal{S}}\longrightarrow G\times_B\mathfrak{m}.$$ The fact that $\widetilde{\epsilon}$ extends $\epsilon$ then tells us that $\overline{\rho}$ extends $\overline{\rho}^{\circ}$. In particular, $\overline{\rho}$ is a meromorphic extension of $\rho$.

We now construct an extension of the inverse
$$(\overline{\rho}^{\circ})^{-1}:(G\times_B\mathfrak{m})^{\circ}\longrightarrow\overline{G\times\mathcal{S}}^{\circ}$$ to a meromorphic map $G\times_B\mathfrak{m}\longrightarrow\overline{G\times\mathcal{S}}$. To this end, recall the closed embedding
$$(\pi_{\mathcal{S}},\overline{\mu}_{\mathcal{S}}):\overline{G\times\mathcal{S}}\longrightarrow\overline{G}\times\g$$ discussed in Remark \ref{Remark: Projective fibres}. One has the composite map
$$\vartheta:=(\pi_{\mathcal{S}}\circ(\overline{\rho}^{\circ})^{-1},\overline{\mu}_{\mathcal{S}}\circ(\overline{\rho}^{\circ})^{-1}):(G\times_B\mathfrak{m})^{\circ}\longrightarrow\overline{G}\times\g.$$ By Lemma \ref{Lemma: Next extension} and \cite[Theorem 3.42]{McKay}, the first component $\pi_{\mathcal{S}}\circ(\overline{\rho}^{\circ})^{-1}$ extends to a meromorphic map
$$\widetilde{\pi_{\mathcal{S}}\circ(\overline{\rho}^{\circ})^{-1}}:G\times_B\mathfrak{m}\longrightarrow\overline{G}.$$ The second component $\overline{\mu}_{\mathcal{S}}\circ(\overline{\rho}^{\circ})^{-1}$ extends to $\nu$ on $G\times_B\mathfrak{m}$, as follows from our earlier observation that $\overline{\mu}_{\mathcal{S}}$ extends $\nu\circ\overline{\rho}^{\circ}$.

Now consider the product meromorphic map
$$\widetilde{\vartheta}:=(\widetilde{\pi_{\mathcal{S}}\circ(\overline{\rho}^{\circ})^{-1}},\nu):G\times_B\mathfrak{m}\longrightarrow\overline{G}\times\g,$$ i.e. the meromorphic map with graph
$$\mathrm{graph}(\widetilde{\vartheta})=\{(\alpha,(\gamma,x))\in (G\times_B\mathfrak{m})\times (\overline{G}\times\g):(\alpha,\gamma)\in\mathrm{graph}(\widetilde{\pi_{\mathcal{S}}\circ(\overline{\rho}^{\circ})^{-1}})\text{ and }\nu(\alpha)=x\}.$$ One can then use direct analogues of arguments made about $\widetilde{\epsilon}$ to establish the following: $\widetilde{\vartheta}$ extends $\vartheta$, and it may be regarded as a meromorphic map $$\overline{\rho^{-1}}:G\times_B\mathfrak{m}\longrightarrow\overline{G\times\mathcal{S}}$$ that extends $(\overline{\rho}^{\circ})^{-1}$ and $\rho^{-1}$. Since $\overline{\rho}$ extends $\rho$, one deduces that $\overline{\rho}$ and $\overline{\rho^{-1}}$ are inverses.

It remains only to establish that \eqref{Equation: Newer diagram} commutes. We begin by interpreting commutativity in \eqref{Equation: Newest diagram} as the statement that
$$\mathrm{graph}(\nu\circ\rho)=\mathrm{graph}(\overline{\mu}_{\mathcal{S}}\big\vert_{G\times\mathcal{S}}).$$ On the other hand, it is straightforward to establish that
$$\mathrm{graph}(\nu\circ\overline{\rho})=\overline{\mathrm{graph}(\nu\circ\rho)}\subseteq(\overline{G\times\mathcal{S}})\times\g\quad\text{and}\quad\mathrm{graph}(\overline{\mu}_{\mathcal{S}})=\overline{\mathrm{graph}({\overline{\mu}_{\mathcal{S}}\big\vert_{G\times\mathcal{S}}})}\subseteq (\overline{G\times\mathcal{S}})\times\g.$$ These last two sentences imply that
$$\mathrm{graph}(\nu\circ\overline{\rho})=\mathrm{graph}(\overline{\mu}_{\mathcal{S}}),$$ i.e. \eqref{Equation: Newer diagram} commutes. This completes the proof. 
\end{proof}

\subsection{The fibre of $\overline{\mu}_{\mathcal{S}}$ over $0$}
One is naturally motivated to improve Theorem \ref{Theorem: GxBHKostantWhittaker}. Section \ref{Subsection: sl2} represents a step in this direction, and it uses the next few results. To this end, recall notation, conventions, and facts discussed in Sections \ref{Subsection: Lie-theoretic conventions} and \ref{Subsection: Wonderful}.

\begin{prop}\label{Proposition: Null fibre}
Suppose that $\gamma\in\overline{G}$. One then has
$$(0,\xi)\in\gamma\Longleftrightarrow \gamma\in (G\times B^{-})\mathfrak{b}\times_{\mathfrak{t}}\mathfrak{b}^{-}.$$
\end{prop}

\begin{proof}
To prove the forward implication, we assume that $(0,\xi)\in\gamma$. One may write $$\gamma=(g_1,g_2)\cdot\mathfrak{p}_I\times_{\mathfrak{l}_I}\mathfrak{p}_I^{-}$$ for suitable elements $(g_1,g_2)\in G\times G$ and a subset $I\subseteq\Pi$. It follows that
$$(0,\mathrm{Ad}_{g_2^{-1}}(\xi))\in\mathfrak{p}_I\times_{\mathfrak{l}_I}\mathfrak{p}_I^{-},$$ or equivalently that
$$\mathrm{Ad}_{g_2^{-1}}(\xi)\in\mathfrak{u}_I^{-}.$$ This combines with \cite[Theorem 5.3]{KostantTDS} and the fact that $\mathrm{Ad}_{g_2^{-1}}(\xi)\in\g^{\text{r}}$ to imply that $I=\emptyset$. We deduce that
$$\gamma=(g_1,g_2)\cdot\mathfrak{b}\times_{\mathfrak{t}}\mathfrak{b}^{-}\quad\text{and}\quad\xi\in\mathrm{Ad}_{g_2}(\mathfrak{b}^{-}).$$ On the other hand, $\mathfrak{b}^{-}$ is the unique Borel subalgebra of $\g$ that contains $\xi$. This implies that $$\mathrm{Ad}_{g_2}(\mathfrak{b}^{-})=\mathfrak{b}^{-},$$ i.e. $g_2\in B^{-}$. We therefore have
$$\gamma=(g_1,g_2)\cdot\mathfrak{b}\times_{\mathfrak{t}}\mathfrak{b}^{-}\in (G\times B^{-})\mathfrak{b}\times_{\mathfrak{t}}\mathfrak{b}^{-}.$$

Now assume that $\gamma\in (G\times B^{-})\mathfrak{b}\times_{\mathfrak{t}}\mathfrak{b}^{-}$, i.e. that
$$\gamma=(g,b)\cdot\mathfrak{b}\times_{\mathfrak{t}}\mathfrak{b}^{-}$$ for some $g\in G$ and $b\in B^{-}$. Since $$\mathrm{Ad}_{b^{-1}}(\xi)\in\mathfrak{u}^{-},$$ we must have
$$(0,\mathrm{Ad}_{b^{-1}}(\xi))\in\mathfrak{b}\times_{\mathfrak{t}}\mathfrak{b}^{-}.$$ This amounts to the statement that
$$(0,\xi)\in (g,b)\cdot\mathfrak{b}\times_{\mathfrak{t}}\mathfrak{b}^{-}=\gamma,$$ completing the proof.
\end{proof}

\begin{cor}\label{Corollary: Flag variety}
There is an isomorphism of varieties
$$\overline{\mu}_{\mathcal{S}}^{-1}(0)\cong G/B.$$
\end{cor}

\begin{proof}
Note that $x_{\mathcal{S}}=\xi$ if $x=0$. It follows that
\begin{align*}
\overline{\mu}_{\mathcal{S}}^{-1}(0) & = \{(\gamma,(0,\xi)):\gamma\in\overline{G}\text{ and }(0,\xi)\in\gamma\}\\
& \cong \{\gamma\in\overline{G}:(0,\xi)\in\gamma\}\\
& = (G\times B^{-})\mathfrak{b}\times_{\mathfrak{t}}\mathfrak{b}^{-}, 	
\end{align*}
where the last line comes from Proposition \ref{Proposition: Null fibre}. At the same time, note that the $(G\times G)$-stabilizer of $\mathfrak{b}\times_{\mathfrak{t}}\mathfrak{b}^{-}\in\overline{G}$ is $B\times B^{-}$ (see \eqref{Equation: StabilisermI}). We therefore have
$$\overline{\mu}_{\mathcal{S}}^{-1}(0)\cong (G\times B^{-})/(B\times B^{-})\cong G/B.$$	
\end{proof}

\begin{rem}\label{Remark: Fibre over 0}
Observe that $\nu^{-1}(0)$ is precisely the zero-section of the vector bundle $G\times_B\mathfrak{m}\longrightarrow G/B$. It follows that $\nu^{-1}(0)\cong G/B$, or equivalently $$\nu^{-1}(0)\cong\overline{\mu}_{\mathcal{S}}^{-1}(0).$$
\end{rem}

\subsection{The case $\g=\mathfrak{sl}_2$}\label{Subsection: sl2}
We now specialize some of our constructions to the case $\g=\mathfrak{sl}_2$. This culminates in Theorem \ref{Theorem: sl2}, to which the following lemma is relevant.

\begin{lem}\label{Lemma: P1}
If $\g=\mathfrak{sl}_2$, then one has a variety isomorphism
$$\overline{\mu}_{\mathcal{S}}^{-1}(x)\cong\mathbb{P}^1$$ for each $x\in\g$.
\end{lem}

\begin{proof}
The case $x=0$ follows from Corollary \ref{Corollary: Flag variety} and the fact that $G/B\cong\mathbb{P}^1$. We may therefore assume that $x\neq 0$. Since $\g=\mathfrak{sl}_2$, this is equivalent to having $x\in\g^{\text{r}}$. Lemma \ref{Lemma: Slice lemma} and the description
$$\overline{\mu}_{\mathcal{S}}^{-1}(x)=\{(\gamma,(x,x_{\mathcal{S}})):\gamma\in\overline{G}\text{ and }(x,x_{\mathcal{S}})\in\gamma\}$$ now imply the following: $\overline{\mu}_{\mathcal{S}}^{-1}(x)$ is isomorphic to the closure of a free orbit of $G_x=G_x\times\{e\}\subseteq G\times G$ in $\overline{G}$. By \cite[Theorem 5.2]{BalibanuPeterson}, the latter is isomorphic to the closure of a free $G_x$-orbit in $G/B\cong\mathbb{P}^1$. Dimension considerations force this new orbit closure to be $\mathbb{P}^1$ itself. The last three sentences imply that $$\overline{\mu}_{\mathcal{S}}^{-1}(x)\cong\mathbb{P}^1.$$
\end{proof}

This facilitates the following improvement to Theorem \ref{Theorem: GxBHKostantWhittaker}.

\begin{thm}\label{Theorem: sl2}
If $\g=\mathfrak{sl}_2$, then the bimeromorphism $\overline{\rho}: \overline{G\times\mathcal S}\to G\times_B\mathfrak{m}$ is a biholomorphism. 
\end{thm}
\begin{proof}
In light of Theorem \ref{Theorem: GxBHKostantWhittaker}, we must prove that $$\overline{\rho}^{\circ}:\overline{G\times\mathcal{S}}^{\circ}\overset{\cong}\longrightarrow(G\times_B\mathfrak{m})^{\circ}$$ extends to a biholomorphism $$\overline{G\times\mathcal{S}}\overset{\cong}\longrightarrow G\times_B\mathfrak{m}.$$ We begin by observing that $\mathfrak{m} = \g$, as $\g=\mathfrak{sl}_2$. The closed embedding \eqref{Equation: Nice closed embedding} is therefore an isomorphism
$$(\pi,\nu):G\times_B\mathfrak{m}\overset{\cong}\longrightarrow G/B\times\g.$$ It thus suffices to prove that $$\varphi:=(\pi,\nu)\circ\overline{\rho}^{\circ}:\overline{G\times\mathcal{S}}^{\circ}\longrightarrow G/B\times\g$$ extends to a biholomorphism
$$\overline{G\times\mathcal{S}}\overset{\cong}\longrightarrow G/B\times\g.$$

Note that the biholomorphism group of $G/B\times\g$ acts transitively, and that $G/B\times \g$ is an unbranched cover of the compact K\"ahler manifold $\mathbb{P}^1\times (\mathbb{C}^3/\mathbb{Z}^6)$. We also know $\varphi$ to be a biholomorphism onto its image, and that $\overline{G\times\mathcal{S}}^{\circ}$ has a complement of codimension at least two in $\overline{G\times\mathcal{S}}$ (see Lemma \ref{Lemma: Codimgeq2}). These last two sentences allow us to apply \cite[Theorem 3.48]{McKay} and deduce that $\varphi$ extends to a local biholomorphism
$$\widetilde{\varphi}:\overline{G\times\mathcal{S}}\longrightarrow G/B\times\g.$$ The commutative diagrams \eqref{Equation: Hessenberg diagram} and \eqref{Equation: Prelim diag} are also easily seen to imply that
$$\begin{tikzcd}
\overline{G\times\mathcal{S}} \arrow{rr}{\widetilde{\varphi}} \arrow[swap]{dr}{\overline{\mu}_{\mathcal{S}}}& & G/B\times\g \arrow{dl}{\mathrm{pr}_{\g}}\\
& \mathfrak{g} &  
\end{tikzcd}$$
commutes, where $\mathrm{pr}_{\g}:G/B\times\g\longrightarrow\g$ is projection to $\g$. It follows that the local biholomorphism $\widetilde{\varphi}$ restricts to a locally injective holomorphic map
$$\widetilde{\varphi}_x:\overline{\mu}_{\mathcal{S}}^{-1}(x)\longrightarrow\mathrm{pr}_{\g}^{-1}(x)$$ for each $x\in\g$. The domain and codomain of $\widetilde{\varphi}_x$ are isomorphic to $\mathbb{P}^1$ (see Lemma \ref{Lemma: P1}), and a locally injective holomorphic map $\mathbb{P}^1\longrightarrow\mathbb{P}^1$ is a biholomorphism. We conclude that $\widetilde{\varphi}_x$ is a biholomorphism for each $x\in\g$. Our local biholomorphism $\widetilde{\varphi}$ is therefore bijective, so that it must be a biholomorphism. This completes the proof.
\end{proof}

\subsection{A conjecture and its implications for Hessenberg varieties}
Theorems \ref{Theorem: GxBHKostantWhittaker} and \ref{Theorem: sl2} and Remark \ref{Remark: Fibre over 0} lend plausibility to the following conjecture.

\begin{conj}\label{Conjecture: Main conjecture}
The bimeromorphism 
$\overline{\rho}:\overline{G\times\mathcal{S}}\overset{\cong}\longrightarrow G\times_B\mathfrak{m}$
in Theorem \ref{Theorem: GxBHKostantWhittaker} is a biholomorphism.
\end{conj}

We now formulate some implications of this conjecture for the geometry of Hessenberg varieties. To this end, assume that Conjecture \ref{Conjecture: Main conjecture} is true. Consider the bundle projection
$\pi:T^*\overline{G}(\log D)\longrightarrow\overline{G}$ and the composite map
$$\psi:=\pi\bigg\vert_{\overline{G\times\mathcal{S}}}\circ\overline{\rho}^{-1}:G\times_B\mathfrak{m}\longrightarrow\overline{G}.$$ Let us write 
$$\psi_x:=\psi\big\vert_{\mathrm{Hess}(x,\mathfrak{m})}:\mathrm{Hess}(x,\mathfrak{m})\longrightarrow\overline{G}$$ for the restriction of $\psi$ to a Hessenberg variety $\mathrm{Hess}(x,\mathfrak{m})\subseteq G\times_B\mathfrak{m}$.

\begin{cor}\label{Corollary: Closed}
Assume that Conjecture \ref{Conjecture: Main conjecture} is true. If $x\in\g$, then $$\psi_x:\mathrm{Hess}(x,\mathfrak{m})\longrightarrow\overline{G}$$ is a closed embedding of algebraic varieties. This embedding is $G_x$-equivariant with respect to the action of $G_x=G_x\times\{e\}\subseteq G\times G$ on $\overline{G}$.
\end{cor}

\begin{proof}
Theorem \ref{Theorem: GxBHKostantWhittaker} and Conjecture \ref{Conjecture: Main conjecture} tell us that $\overline{\rho}$ restricts to a biholomorphism
$$\overline{\rho}_x:\overline{\mu}_{\mathcal{S}}^{-1}(x)\overset{\cong}\longrightarrow\nu^{-1}(x)=\mathrm{Hess}(x,\mathfrak{m})$$
of complex analytic spaces. Since $\overline{\mu}_{\mathcal{S}}^{-1}(x)$ and $\mathrm{Hess}(x,\mathfrak{m})$ are projective (see Remarks \ref{Remark: Projective fibres} and \ref{Remark: Closed embedding}), it follows that $\overline{\rho}_x$ is a variety isomorphism (e.g. by \cite[Corollary A.4]{Lange}). 
At the same time, Remark \ref{Remark: Projective fibres} implies that $\pi$ restricts to a closed embedding
$$\pi_x:\overline{\mu}_{\mathcal{S}}^{-1}(x)\longrightarrow\overline{G},\quad (\gamma,(x,x_{\mathcal{S}}))\mapsto \gamma,\quad (\gamma,(x,x_{\mathcal{S}}))\in \overline{\mu}_{\mathcal{S}}^{-1}(x).$$ The composite map
$$\psi_x=\pi_x\circ(\overline{\rho}_x)^{-1}:\mathrm{Hess}(x,\mathfrak{m})\longrightarrow\overline{G}$$ is therefore a closed embedding of algebraic varieties. The claim about equivariance follows from the $G$-equivariance of $\overline{\rho}$, together with the fact that $\pi$ is ($G\times G$)-equivariant.
\end{proof}

One can be reasonably explicit about the image of $\psi_x$, especially if $x\in\g^{\text{r}}$. To elaborate on this, recall the ($G\times G$)-equivariant locally closed embedding $$\kappa:G\longrightarrow\overline{G}\subseteq\mathrm{Gr}(n,\g\oplus\g)$$ defined in \eqref{Equation: Def of gamma}. Given any $x\in\g$, let us write $\overline{G_x}$ for the closure of $\kappa(G_x)$ in $\overline{G}$.

\begin{prop}\label{Proposition: momentfibre}
Assume that Conjecture \ref{Conjecture: Main conjecture} is true. If $x\in\mathfrak{g}$, then the following statements hold.
\begin{itemize}
\item[(i)] We have $$\mathrm{image}(\psi_x)=\{\gamma\in\overline{G}:(x,x_{\mathcal{S}})\in\gamma\}.$$
\item[(ii)] If $x\in\g^{\emph{r}}$, then $$\mathrm{image}(\psi_x)=(e,g^{-1})\cdot\overline{G_x}$$ for any $g\in G$ satisfying $x=\mathrm{Ad}_g(x_{\mathcal{S}})$.
\end{itemize}
\end{prop}

\begin{proof}
To prove (i), recall the notation used in the proof of Corollary \ref{Corollary: Closed}. The image of $\psi_x$ coincides with that of $\pi_x$, and the latter is
$$\{\gamma\in\overline{G}:(x,x_{\mathcal{S}})\in\gamma\}.$$ 
	
We now verify (ii). Note that
$$(x,x_{\mathcal{S}})=(\mathrm{Ad}_{hg}(x_{\mathcal{S}}),x_{\mathcal{S}})\in\gamma_{hg}$$ 
for all $h\in G_x$. This combines with (i) to imply that $\gamma_{hg}\in\mathrm{image}(\psi_x)$ for all $h\in G_x$, i.e. $$\kappa(G_xg)\subseteq\mathrm{image}(\psi_x).$$ Since $\mathrm{image}(\psi_x)$ is closed in $\overline{G}$, it must therefore contain the closure $\overline{\kappa(G_xg)}$ of $\kappa(G_xg)$ in $\overline{G}$. We also note that $\mathrm{image}(\psi_x)\cong\mathrm{Hess}(x,\mathfrak{m})$ is $\ell$-dimensional and irreducible (see \cite[Corollaries 3 and 14]{PrecupTransformation}), and that $$\dim(\overline{\kappa(G_xg)})=\dim(G_x)=\ell.$$ It follows that $$\overline{\kappa(G_xg)}=\mathrm{image}(\psi_x).$$ It remains only to invoke the ($G\times G$)-equivariance of $\kappa$ and conclude that
$$\overline{\kappa(G_xg)}=(e,g^{-1})\cdot\overline{\kappa(G_x)}=(e,g^{-1})\cdot\overline{G_x}.$$
\end{proof}

\begin{rem}
Assume that Conjecture \ref{Conjecture: Main conjecture} is true. If $x\in\mathcal{S}$, then one may apply Proposition \ref{Proposition: momentfibre}(ii) with $g=e$ and conclude that
$\mathrm{image}(\psi_x)=\overline{G_x}$. Corollary \ref{Corollary: Closed} therefore yields a $G_x$-equivariant isomorphism
$$\mathrm{Hess}(x,\mathfrak{m})\overset{\cong}\longrightarrow\overline{G_x}.$$ This is precisely the isomorphism obtained in \cite[Corollary 4.10]{BalibanuUniversal}.
\end{rem} 
\section*{Notation}
\begin{itemize}
\item $\g$ --- finite-dimensional complex semisimple Lie algebra
\item $n$ --- dimension of $\g$
\item $\ell$ --- rank of $\g$
\item $\g^{\text{r}}$ --- subset of regular elements in $\g$
\item $V^{\text{r}}$ --- the intersection $V\cap\g^{\text{r}}$ for any subset $V\subseteq\g$
\item $\g^{\text{rs}}$ --- subset of regular semisimple elements in $\g$
\item $\langle\cdot,\cdot\rangle$ --- Killing form on $\g$
\item $V^{\perp}$ --- annihilator of $V\subseteq\g$ with respect to $\langle\cdot,\cdot\rangle$
\item $G$ --- adjoint group of $\g$
\item $G_x$ --- $G$-stabilizer of $x\in\g$ with respect to the adjoint action
\item $\mu=(\mu_L,\mu_R):T^*G\longrightarrow\g\oplus\g$ --- moment map for the ($G\times G$)-action on $T^*G$
\item $\mu_{\mathcal{S}}:G\times\mathcal{S}\longrightarrow\g$ --- restriction of $\mu_L$ to $G\times\mathcal{S}$
\item $\tau$ ---  fixed principal $\mathfrak{sl}_2$-triple in $\g$
\item $\mathcal S$ --- Slodowy slice associated to $\tau$
\item $\chi:\g\longrightarrow\mathrm{Spec}(\mathbb{C}[\g]^G)$ --- adjoint quotient 
\item $x_{\mathcal S}$ --- unique point at which $\mathcal{S}$ meets the fibre $\chi^{-1}(\chi(x))$
\item $\overline{G}$ --- De Concini--Procesi wonderful compactification of $G$
\item $\g_\Delta$ --- diagonal in $\g\oplus\g$
\item $\gamma_g$ --- the point $(g,e)\cdot\g_\Delta\in\overline{G}$
\item $\kappa:G\longrightarrow\overline{G}$ --- open embedding defined by $g\mapsto\gamma_g$
\item $D$ --- the divisor $\overline{G}\setminus \kappa(G)$
\item $\overline{G_x}$ --- closure of $\kappa(G_x)$ in $\overline{G}$
\item $T^*\overline{G}(\log(D))$ --- log cotangent bundle of $(\overline{G},D)$
\item $\overline{\mu}=(\overline{\mu}_L,\overline{\mu}_R):T^*\overline{G}(\log(D))\longrightarrow\g\oplus\g$ --- moment map for the ($G\times G$)-action on $T^*\overline{G}(\log(D))$
\item $\overline{G\times\mathcal{S}}$ --- the Poisson slice $\overline{\mu}_R^{-1}(\mathcal{S})$ 
\item $\overline{\mu}_{\mathcal{S}}:\overline{G\times\mathcal{S}}\longrightarrow\g$ --- restriction of $\overline{\mu}_L$ to $\overline{G\times\mathcal{S}}$
\item $\mathcal{Z}_{\mathfrak{g}}$ --- universal centralizer of $\g$
\item $\overline{\mathcal{Z}_{\g}}$ --- B\u{a}libanu's fibrewise compactification of $\mathcal{Z}_{\g}$
\item $\mathcal B$ --- abstract flag variety of all Borel subalgebras in $\g$
\item $\mathfrak{m}$ --- standard Hessenberg subspace
\item $\nu:G\times_B\mathfrak{m}\longrightarrow \g$ --- moment map for the $G$-action on $G\times_B\mathfrak{m}$
\item $\mathrm{Hess}(x,\mathfrak{m})$ --- fibre of $\nu$ over $x\in\g$ 
\item $\Pi$ --- set of simple roots
\item $\p_I,\p_I^-$ --- standard parabolic, opposite parabolic associated to $I\subseteq\Pi$
\item $\lf_I$ --- standard Levi subalgebra $\p_I\cap\p_I^-$ associated with $I\subseteq \Pi$ 
\end{itemize}

\bibliographystyle{acm} 
\bibliography{MFSreg}
\end{document}